\documentclass[a4paper,12pt]{article}
\usepackage{graphics}
\usepackage{pgf,tikz}
\usetikzlibrary{arrows}

\usepackage{amssymb}
\usepackage[colorlinks=true
linkscolor=red, urlcolor=cyan]{hyperref}
\usepackage{graphicx}
\usepackage{here}
\usepackage{color}
\newtheorem{theorem}{\textbf{Theorem}}
\newtheorem{lemma}{\textbf{Lemma}}

\newtheorem{remark}{\textbf{Remark}}
\newtheorem{proof}{\textbf{Proof}}
\begin{document}

\title{An impulsive modelling framework of fire occurrence in a size structured model of tree-grass interactions for savanna ecosystems}
\author{V. Yatat$^{1,2}$, P. Couteron$^3$, J.J. Tewa$^{1,2}$, S. Bowong$^{4,2}$, and Y. Dumont$^5$\footnote{Corresponding author: yves.dumont@cirad.fr} \\
$^1$University of Yaounde I, LIRIMA, GRIMCAPE team, Cameroon\\
$^2$IRD, UMI 209, UMMISCO, IRD France Nord, F-93143, Bondy, France\\
$^3$IRD, Umr AMAP, Montpellier, France \\
$^4$University of Douala, LIRIMA, GRIMCAPE team, Cameroon \\
$^5$CIRAD, Umr AMAP, Montpellier, France}
%
%
\maketitle
\begin{abstract}
Fires and rainfall are major mechanisms that regulate woody and
grassy biomasses in savanna ecosystems. Conditions of long-lasting
coexistence of trees and grasses have been mainly studied using
continuous-time modelling of tree-grass competition. In these
frameworks, fire is a time-continuous forcing while the relationship
between woody plant size and fire-sensitivity is not systematically
considered. In this paper, we propose a new mathematical framework
to model tree-grass interaction that takes into account both the
discrete nature of fire occurrence and size-dependent fire
sensitivity (via two classes of woody plants). We carry out a
qualitative analysis that highlights ecological thresholds and
bifurcations parameters that shape the dynamics of the savanna-like
systems within the main ecological zones. Moreover, through a 
qualitative analysis, we show that the impulsive modelling of fire
occurrences leads to more diverse behaviors and a more realistic array of solutions than the analogous
time-continuous fire models. Numerical simulations are provided to illustrate the theoretical results 
and to support a discussion about the bifurcation parameters and future developments.
\end{abstract}
\paragraph*{key words:}Asymmetric competition -- Savanna -- Fire --
Impulsive differential equation -- Qualitative analysis -- Nonstandard finite difference
scheme
\section{Introduction}
Savannas are ecosystems with fairly continuous grass cover and
variable woody cover (Maurin et al. (2014) \cite{Maurin2014}).
 However, savanna-like ecosystems are diverse and cover extensive areas
throughout the tropics. Explanations found in the literature about
the possible long-lasting coexistence of woody and grassy vegetation
components therefore relate to diverse factors and processes
depending on the location and the ecological context (Baudena et al.
(2014) \cite{Baudena2014}). Several studies have pointed towards the
role of stable ecological factors e.g. climate, in shaping the tree
to grass ratio along large-scale gradients of rainfall or soil
fertility (Sankaran et al. (2005) \cite{Sankaran2005}, (2008)
\cite{Sankaran2008}). Other studies have rather emphasized the
reaction of vegetation to recurrent disturbances such as herbivory
or fire (Langevelde et al. (2003) \cite{Langevelde2003}, D'Odorico
et al. (2006) \cite{Dodorico2006}, Sankaran et al. (2008)
\cite{Sankaran2008}, Smit et al. (2010) \cite{Smit2010}, Favier et
al. (2012) \cite{Favier2012} and references therein). Those two
points of view are not mutually-exclusive since both environmental
control and disturbances may co-occur in a given area and along
ecological gradients, although their relative importance generally
varies among ecosystems. Bond et al. (2003) \cite{Bond2003} proposed
the name of climate-dependent for ecosystems which physiognomies are
highly dependent on climatic conditions (rainfall, soil moisture)
versus disturbance-dependent for ecosystems which dynamics is
strongly dependent on fires or herbivores.
\par

 Several
models using a system of ordinary differential equations (ODES) have
been proposed to depict and understand the dynamics of woody and
herbaceous components in savanna-like vegetation. A first attempt
(Walker et al. (1981) \cite{Walker1981}) was orientated towards
semiarid fireless savannas and analyzed the effect of herbivory and
drought on the balance between woody and herbaceous biomass. This
model refers to ecosystems immune to fire due to insufficient annual
rainfall and grass production. Indeed, fires in savanna-like
ecosystems mostly rely on herbaceous biomass that has dried up
during the dry season. As long as rainfall is sufficient, fires
impact seedlings and saplings within the flame zone and thus let
grasses indirectly inhibit tree establishment.

\par
More recently, several attempts have been made (see Langevelde et
al. (2003) \cite{Langevelde2003}, Accatino et al. (2010)
\cite{Accatino2010}, De Michele et al. (2011) \cite{Demichele2011},
Tchuinte et al. (2014) \cite{Tchuinte2014}, Yatat et al. (2014)
\cite{Yatat2014}) to model the dynamics of savannas, taking into
account fires as continuous events, on the basis of the initial
framework of Tilman (1994) \cite{Tilman1994} that used coupled
 ODES to model the competitive interactions between two kinds of plants.
\par
 However, it is questionable to
model fire as a continuous forcing that continuously removes
fractions of fire sensitive biomass. Indeed, several months and even
years can past between two successive fires, such that fire may be
considered as an instantaneous perturbation of the savanna
ecosystem. Several recent papers have proposed to model fires as
stochastic events while keeping the continuous-time differential
equation framework (Baudena et al. (2010) \cite{Baudena2010},
Beckage et al. (2011) \cite{Beckage2011}) or using time discrete
matrix models (Accatino \& De Michele (2013) \cite{Accatino2013}).
But in all those examples, fire characteristics remain mainly a
linear function of grass biomass which is not satisfactory. Indeed,
it is well known that at low grass biomass there is no fires while
above a sufficient grass biomass, fires intensity increases rapidly
before reaching a saturation. This particular feature of fires in
savanna vegetation cannot be modeled by a linear function. Another
drawback of the aforementioned recent stochastic models (Baudena et
al. (2010) \cite{Baudena2010}, Beckage et al. (2011)
\cite{Beckage2011}) is that they barely lend themselves to
analytical (qualitative) approaches.
\par
 In this paper, we therefore
present a model that differs from most published models and extends
the work of Yatat et al. (2014) \cite{Yatat2014} by modelling
discrete fire occurrences. We consider a tree-grass compartmental
model. We set one compartment for grass biomass and two for trees,
namely fire-sensitive individuals having most of their buds within
the flame zone (like seedlings, saplings, shrubs) and non-sensitive
mature trees having at least their upperparts above the flame zone.
We therefore develop a system of three coupled non-linear impulsive
differential equations (IDES), one equation per vegetation
compartment, that describes savanna dynamics. In addition, we model
fire intensity (and the corresponding impact on sensitive woody
plants) as an increasing nonlinear and bounded function of grass
biomass. Finally, fire occurrences are modeled as pulse-like
perturbations.
\par
In order to assess and illustrate the dynamics of some ecological
formations through our mathematical model, we distinguish in our
numerical computations three climatic zones having distinct
characteristics in terms of biomass production. These biomass
production zones, loosely relate to Africa and are indexed by their
carrying capacity for grass and woody biomasses and by fires return
times. 
Semi-arid areas have a mean annual rainfall that varies between 300
$mm. yr^{-1}$ and 650 $mm. yr^{-1}$, and fires, if any have return
intervals relatively long, says less than one fire event every ten
years. Mesic savannas have a mean annual rainfall comprised between
650 $mm. yr^{-1}$ and 1100 $mm. yr^{-1}$ and fire return time is in
order of four or five years, sometimes less. Finally, we consider a
humid tropical area with a mean annual rainfall between 1100 $mm.
yr^{-1}$--1800 $mm. yr^{-1}$ and where one can have a fire return
time from less than one year to two-three years.
\par
Although impulsive differential equations appear highly relevant to
model vegetation dynamics in fire prone savannas, they are also
difficult in terms of analytical treatments. This may explain why
this framework has still remained scarcely used for modelling the
dynamics of fire prone savannas. Our model aims to acknowledge three
major phenomena: the periodic occurrence of fire events, the
fire-mediated, non-linear negative feedback of grasses onto
sensitive trees and the negative and/or positive feed-back,
depending on location, of insensitive trees on grasses. We therefore
explicitly model the occurrence of fires in savanna ecosystems and
the asymmetric nature of tree-grass competitive interactions in
savannas.

\par
The full impulse fire model of asymmetric tree-grass competition
(IFAC) is formulated in Section \ref{formulation}. In Section
\ref{QA} we reach qualitative analytical results for IFAC through
which we highlight some meaningful ecological thresholds that
summarize savanna dynamics under impulsive fires. 
We present a nonstandard numerical scheme
for the IFAC model in Section \ref{NS} together with the IFAC
parameters ranges. In section \ref{discussion} we present numerical
simulations done in the three ecological biomass production areas of
the African continent. Moreover, Section \ref{discussion} also deals
with the discussion of our results.

\section{The impulse fire model of asymmetric tree-grass competition (IFAC)
formulation}\label{formulation}
 As we have mentioned before, we
consider vegetation as composed of three classes, i.e. the class of
sensitive tree biomass $(T_S)$, the class of non-sensitive tree
biomass $(T_{NS})$ and the class of grass biomass $(G).$ We model
the fire intensity by a non-linear increasing function of grass
biomass $w(G)$. To built up our model, we consider the following
assumptions where  (A4), (A5), (A6) and (A7) are already described
in Yatat et al. (2014) \cite{Yatat2014}.
\begin{itemize}
    \item[\textbf{(A1)}] A carrying capacity $K_T$ for tree biomass (in tons per
    hectare, $t.ha^{-1}$).
    \item[\textbf{(A2)}] A carrying capacity $K_G$ for grass biomass (in tons per
    hectare, $t.ha^{-1}$).
    \item[\textbf{(A3)}] Fire events occur periodically, i.e. every $\tau$-time,
    where $\tau=\displaystyle\frac{1}{f}$ and $f$ denotes the fire frequency.
    \item[\textbf{(A4)}] Fire only impacts grass and sensitive Tree and, fire intensity is an increasing function of the grass biomass
    \item[\textbf{(A5)}]The Grass biomass has a direct, depressing effect on the Sensitive Trees.
    \item[\textbf{(A6)}]Non Sensitive Trees have a depressive or facilitation effect on grass biomass by shading.
    \item[\textbf{(A7)}]Sensitive tree biomass moves to non-sensitive tree biomass after an averaged
    time $\displaystyle\frac{1}{\omega_S}$ (in years).
    \end{itemize}
The following diagram summarize the relationship between the three
compartments\vspace{-1cm}
\begin{figure}[H]
\begin{center}
\definecolor{qqqqff}{rgb}{0,0,1}
\begin{tikzpicture}[line cap=round,line join=round,>=triangle 45,x=1.0cm,y=1.0cm,scale=0.9]
\clip(-2.62,-7.8) rectangle (29.06,6.1); \draw [line
width=1.2pt,color=qqqqff] (-1,4)-- (-1,2); \draw [line
width=1.2pt,color=qqqqff] (-1,2)-- (1,2); \draw [line
width=1.2pt,color=qqqqff] (1,2)-- (1,4); \draw [line
width=1.2pt,color=qqqqff] (1,4)-- (-1,4); \draw [line
width=1.2pt,color=qqqqff] (7,4)-- (7,2); \draw [line
width=1.2pt,color=qqqqff] (7,2)-- (9,2); \draw [line
width=1.2pt,color=qqqqff] (9,2)-- (9,4); \draw [line
width=1.2pt,color=qqqqff] (9,4)-- (7,4); \draw [rotate
around={0:(4,-1.93)},line width=1.2pt] (4,-1.93) ellipse (2cm and
1.07cm); \draw [shift={(-1.42,1.5)},line width=1.2pt]
plot[domain=-1.04:0.34,variable=\t]({1*1.51*cos(\t r)+0*1.51*sin(\t
r)},{0*1.51*cos(\t r)+1*1.51*sin(\t r)}); \draw
[shift={(1.67,-1.99)},line width=1.2pt]
plot[domain=0.76:5,variable=\t]({1*0.61*cos(\t r)+0*0.61*sin(\t
r)},{0*0.61*cos(\t r)+1*0.61*sin(\t r)}); \draw
[shift={(2.95,-3.4)},line width=1.2pt,dash pattern=on 5pt off 5pt]
plot[domain=-1.02:0.58,variable=\t]({1*0.77*cos(\t r)+0*0.77*sin(\t
r)},{0*0.77*cos(\t r)+1*0.77*sin(\t r)}); \draw
[shift={(6.08,-3.26)},line width=1.2pt]
plot[domain=2.87:3.74,variable=\t]({1*1.31*cos(\t r)+0*1.31*sin(\t
r)},{0*1.31*cos(\t r)+1*1.31*sin(\t r)}); \draw
[shift={(0,4.26)},line width=1.2pt]
plot[domain=0.02:3.54,variable=\t]({1*0.68*cos(\t r)+0*0.68*sin(\t
r)},{0*0.68*cos(\t r)+1*0.68*sin(\t r)}); \draw [->,line
width=1.2pt] (-1,3) -- (-2.46,3); \draw [->,line width=1.2pt]
(1,2.54) -- (7,2.54); \draw [->,line width=1.2pt] (6,-2) --
(7.5,-2.02); \draw [line width=1.2pt] (7.78,2)-- (6.58,-2.01); \draw
[line width=1.2pt] (2.44,-1.25)-- (-0.32,0.47); \draw [->,line
width=1.2pt] (0.68,4.28) -- (0.7,4); \draw [->,line width=1.2pt]
(-0.66,0.2) -- (-1,0); \draw [->,line width=1.2pt] (3.36,-4.06) --
(3.06,-4.32); \draw [->,line width=1.2pt] (5,-4) -- (5.26,-4.38);
\draw [->,line width=1.2pt] (1.84,-2.58) -- (2.21,-2.41); \draw
(1.72,3.94) node[anchor=north west] {$ \gamma_{NS}$}; \draw
(-0.56,3.46) node[anchor=north west] {\Large{$T_S$}}; \draw
(3.46,-1.64) node[anchor=north west] {\Large{$G$}}; \draw
(7.16,3.46) node[anchor=north west] {\Large{$T_{NS}$}}; \draw
(3.16,2.46) node[anchor=north west] {$\omega_S$}; \draw (-0.58,5.56)
node[anchor=north west] {$ \gamma_{S}$}; \draw (0.16,-1.76)
node[anchor=north west] {$ \gamma_{G}$}; \draw (5.18,-3.46)
node[anchor=north west] {$\mu_{G}$}; \draw [->,line width=1.2pt]
(9,3) -- (10.5,2.96); \draw (9.3,3.68) node[anchor=north west]
{$\mu_{NS}$}; \draw (-2.16,3.7) node[anchor=north west] {$\mu_{S}$};
\draw (-2.06,0.56) node[anchor=north west] {$\sigma_GG  $}; \draw
(7.4,-1.6) node[anchor=north west] {$\sigma_{NS} T_{NS}$}; \draw
(2.18,1.36) node[anchor=north west] {$ \eta_Sw(G)$}; \draw
(2.42,-3.44) node[anchor=north west] {$\eta_G$}; \draw
[shift={(1.72,2.16)},line width=1.2pt,dash pattern=on 5pt off 5pt]
plot[domain=3.27:4.95,variable=\t]({1*1.19*cos(\t r)+0*1.19*sin(\t
r)},{0*1.19*cos(\t r)+1*1.19*sin(\t r)}); \draw [->,line
width=1.2pt] (2,1) -- (2.3,1.08); \draw [shift={(5.53,1.42)},line
width=1.2pt]  plot[domain=0.96:2.17,variable=\t]({1*2.58*cos(\t
r)+0*2.58*sin(\t r)},{0*2.58*cos(\t r)+1*2.58*sin(\t r)}); \draw
[shift={(2.58,5.71)},line width=1.2pt]
plot[domain=4.24:5.32,variable=\t]({1*2.62*cos(\t r)+0*2.62*sin(\t
r)},{0*2.62*cos(\t r)+1*2.62*sin(\t r)}); \draw [->,line
width=1.2pt] (1.38,3.38) -- (1,3.6); \draw [line width=1.2pt,dash
pattern=on 5pt off 5pt] (2.76,-1.09)-- (1.22,1.15);
\end{tikzpicture}
\end{center}
\vskip-100pt
 \caption{Compartmental diagram of the size structured tree-grass interaction model in impulse fires-prone savanna.}
\label{diagramme}
\end{figure}
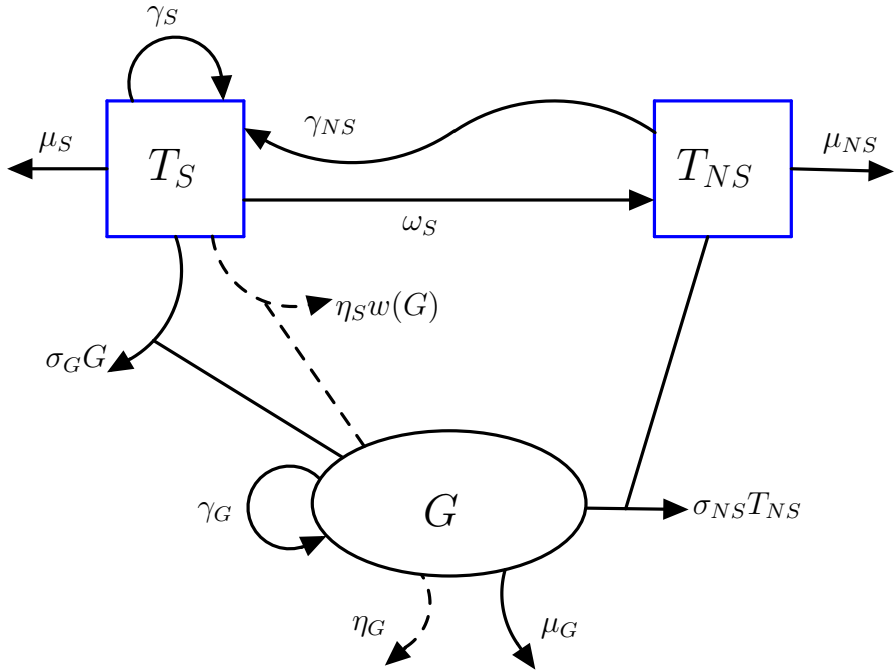

The following parameters are considered throughout the paper:
\begin{itemize}
    \item Sensitive tree biomass is made up from non sensitive tree
    biomass (i.e. seed production and germination) with the rate $\gamma_{NS}~(in~yr^{-1})$ and from existing sensitive
    tree biomass (i.e. intrinsic growth) with the rate $\gamma_S~(in~yr^{-1})$.
    \item Grass biomass is made up from existing grass biomass with the intrinsic growth rate
    $\gamma_{G}~(in~yr^{-1})$.
    \item $\mu_S~(in~yr^{-1})$ is an additional death rate  of sensitive tree biomass due to external disturbances such as human activities and herbivory.
    \item $\mu_{NS}~(in~yr^{-1})$ is the natural death rate  of non sensitive tree biomass.
    \item $f$ is the fire frequency (in $yr^{-1}$).
    \item $\mu_G~(in~yr^{-1})$ is an additional death rate of grass biomass due to factors including human activities and herbivory.
    \item $\displaystyle\frac{1}{\omega_S}$ (in $yr$) is the average time that a sensitive tree takes to become
     non sensitive to fire.
    \item $\sigma_G$ expresses the asymmetric competition exerted by grasses on sensitive trees (shading and competition for nutrients, in
    $ha.t^{-1}.yr^{-1}$).
    \item $\sigma_{NS}$ expresses the asymmetric interaction (competition and/or facilitation) of non sensitive trees on grasses (shading and competition for nutrients, in
    $ha.t^{-1}.yr^{-1}$).
    \item $\eta_{S}$ is the proportion of sensitive tree biomass that is consumed by fire.
    \item $\eta_{G}$ is the proportion of grass biomass that is consumed by fire.
\end{itemize}

Based on these ecological premises,  and  taking into account  the
effect of fire as pulse phenomena, we propose a model for the
savanna vegetation dynamics which is and extension of the model
formulated in Yatat et al. (2014) \cite{Yatat2014}.

The IFAC is given by
\begin{equation}\label{fa}
 \left\{%
\begin{array}{lclcr}
  \displaystyle\frac{dT_S}{dt} &=& (\gamma_ST_S+\gamma_{NS}T_{NS})\left(1-\displaystyle\frac{T_S+T_{NS}}{K_T}\right)-T_{S}(\mu_S+\omega_S+\sigma_G G),& & \\
  & & & & \\
  \displaystyle\frac{dT_{NS}}{dt} &=& \omega_ST_S-\mu_{NS}T_{NS},& & t\neq t_k\\
  & & & & \\
  \displaystyle\frac{dG}{dt} &=&
  \gamma_G\left(1-\displaystyle\frac{G}{K_G}\right)G-(\sigma_{NS}T_{NS}+\mu_G)G, & &
\end{array}
\right.
\end{equation}

\begin{equation}\label{ifa}
 \left\{%
\begin{array}{lclcr}
  T_S(t_k^+) &=& (1-\eta_Sw(G(t_k)))T_S(t_k),& & \\
  T_{NS}(t_k^+) &=& T_{NS}(t_k),&t=t_k & t_{k+1}=t_k+\tau\\
  G(t_k^+) &=& (1-\eta_G)G(t_k), & &
\end{array}
\right.
\end{equation}
with
\begin{equation}\label{ci}
  T_S(0) = T_{S_0}>0,~T_{NS}(0) =T_{NS_0}\geq0~,  ~ G(0) = G0>0.
\end{equation}
For this impulsive fire model, the fire intensity function $w$ is a
continuous and positive function of grass biomass which is bounded
above by unity. As in Tchuinte et al. (2014) \cite{Tchuinte2014} and
Yatat et al. (2014) \cite{Yatat2014}, one can choose a generic
sigmoidal function (see also Staver et al. (2011)
\cite{Staver2011}). A typical choice could be
    \begin{equation}
   w(G)=\displaystyle\frac{G^\alpha}{G^\alpha+g_0^\alpha},
   \end{equation}
   where $G_0=g_0^\alpha$ is the value of grass biomass at which fire intensity reaches its
   half saturation ($g_0$ in tons per
    hectare, $t.ha^{-1}$) and $\alpha\in \mathbb{N}^*$.\\
The feasible region for system $(\ref{fa})-(\ref{ifa})$ is the set
$\Omega$ defined by
\begin{equation}
    \begin{array}{ccl}
      \Omega & = & \{(T_S; T_{NS}; G) \in \mathbb{R}^3_+~|~0\leq T_S+T_{NS}\leq K_T; 0\leq G\leq
      K_G\}.
    \end{array}
\end{equation}

\section{Mathematical Analysis and Ecological Interpretation of Thresholds}\label{QA}

\subsection{Existence of solution} The right-hand side of
system (\ref{fa})-(\ref{ifa}) is locally lipschitz continuous on
$\Omega$. Thus, using a classic existence Theorem (Theorem 1.1 P. 3
in Bainov and Simeonov (1995) \cite{Bainov1995}), system
(\ref{fa})-(\ref{ifa})-(\ref{ci}) has a unique solution on $\Omega$.

\subsection{Trivial and semitrivial solutions} It is obvious that
system $(\ref{fa})-(\ref{ifa})$ has always a desert equilibrium\\
$E_0=(0,0,0)$.

\subsubsection{The positive grassland periodic solution: existence
and local stability\\}

Let us consider the following thresholds:
\begin{equation}
\mathcal{R}_G^0=\displaystyle \frac{\gamma_G}{\mu_G}, \qquad\mu_G>0,
\label{seuilR20}
\end{equation}

and
\begin{equation}
\rho_{G}^0=\left\{
\begin{array}{ll}
\displaystyle (1-\eta_G)\exp\left(\mu_G\left(\mathcal{R}_G^0-1\right) \tau \right), \qquad &\mu_G>0, \\
(1-\eta_G)\exp\left(\gamma_G\tau\right), \qquad &\mu_G=0.
\end{array}
\right. \label{seuilR20tauetaG}
\end{equation}
Assume that $\mathcal{R}_G^0>1 $, we have the following result (see
also Dai et al. (2012) \cite{Dai2012})
\begin{lemma}\label{existenceEG}
When $\rho_{G}^0>1$, System
$(\ref{fa})-(\ref{ifa})$ has a positive grassland periodic solution
$E_G=(0; 0; G^*(t))$, where

\begin{equation}\label{getoile}
G^*(t)=\left\{
\begin{array}{ccc}
  \displaystyle\frac{\displaystyle
K_G\left(1-\displaystyle\frac{1}{\mathcal{R}_G^0}\right)
\left(\rho_G^0-1\right)}{(\rho_G^0-1)+\eta_Ge^{-\mu_G\left(\mathcal{R}_G^0-1\right)(t-(n+1)\tau)}}, & \mu_G>0, &  \\
   &  & n\tau\leq
t<(n+1)\tau, \\
 \displaystyle
\frac{\displaystyle
K_G\left(\rho_G^0-1\right)}{(\rho_G^0-1)+\eta_Ge^{-\gamma_G(t-(n+1)\tau)}},
& \mu_G=0. &
\end{array}
\right.
\end{equation}

\end{lemma}
\begin{remark}[Thresholds interpretation]
\begin{itemize}
    \item[$\bullet$] $\mathcal{R}_G^0$ is the average amount of biomass produced per unit of grass biomass during its
whole lifespan in absence of fires and depressing effect from non
sensitive trees but subject to additional mortality caused by human
activities or by herbivory.
    \item[$\bullet$] $\rho_{G}^0$ embodies the residual amount of grass biomass at any time-period that
fires occur, from the grass biomass produced per unit of grass
biomass.
\end{itemize}
\end{remark}
\begin{remark}
An obvious computation leads to:
\begin{equation}\label{seuil_etaG}
    \begin{array}{cclc}
      \mathcal{R}_G^0<1 & \Longrightarrow & \rho_G^0<1,& \\
      \rho_G^0>1 & \Longleftrightarrow &
    \eta_G<1-\exp(-\mu_G(\mathcal{R}_G^0-1)\tau),&\mu_G>0,\\
    \rho_G^0>1 & \Longleftrightarrow &
    \eta_G<1-\exp(-\gamma_G\tau),&\mu_G=0.
    \end{array}
\end{equation}
\end{remark}

Let us state the following Lemma which will be helpful for the
sequel
\begin{lemma}\label{lemme_integrale}
When $\rho_{G}^0>1$,
\begin{equation}\label{integrale1}
    \begin{array}{llll}
     G_{int}:= &\displaystyle \frac{1}{\tau}\int_{n\tau}^{(n+1)\tau}G^*(s)ds & = &
      \displaystyle\frac{1}{\tau}\displaystyle\frac{K_G}{\gamma_G}\ln(\rho_G^0)=\frac{1}{\tau}\displaystyle\frac{K_G}{\gamma_G}
      \left(\ln(1-\eta_G)+(\gamma_G-\mu_G)\tau\right)>0.
    \end{array}
\end{equation}
\end{lemma}

Now we turn to look for local stability of that previous positive
and periodic grassland solution. For that purpose, we will use the
small perturbation technique and the Floquet theory, i.e. we will
find conditions under which all the Floquet multipliers of the
positive and periodic grassland solution have their absolute value
less than unity or equal to unity (D'Onofrio (2002)
\cite{Donofrio2002}, Chen
et al. (2009) \cite{Chen2009}).\\
Let
\begin{equation}\label{seuilR2T}
\begin{array}{ccl}
\mathcal{R}_G^T&=&\displaystyle\frac{\gamma_S\mu_{NS}+\omega_S\gamma_{NS}}{\mu_{NS}(\mu_S+\omega_S)+\mu_{NS}\sigma_GG_{int}},
 \\
 &&\\
   & = & \displaystyle\frac{\gamma_S\mu_{NS}+\omega_S\gamma_{NS}}{\mu_{NS}(\mu_S+\omega_S)+\displaystyle\frac{\mu_{NS}\sigma_G}{\tau}\displaystyle\frac{K_G}{\gamma_G}\left(\ln(1-\eta_G)+(\gamma_G-\mu_G)\tau\right)},
\end{array}
\end{equation}
$$
  \mathcal{A}=\gamma_S\tau\left(1-\displaystyle\frac{1}{\mathcal{R}}\right),
$$
where
$$\mathcal{R}=\displaystyle\frac{\gamma_S}{\mu_S+\omega_S+\mu_{NS}+\sigma_GG_{int}}$$
and
$$
  \mathcal{B} =
  \tau^2\mu_{NS}\left(\mu_S+\omega_S+\sigma_GG_{int}\right)\left(1-\mathcal{R}_G^T\right).
$$
Moreover, let $\lambda_1$, $\lambda_2$ be the roots of
$$\mathcal{P}(\lambda)=\lambda^2-\mathcal{A}\lambda+\mathcal{B}$$ and

\begin{equation}\label{rhot}
\rho_T=\max\{(1-\eta_Sw(G^*(\tau)))e^{\lambda_1},~e^{\lambda_2}\}.
\end{equation}

The following result holds for System $(\ref{fa})-(\ref{ifa})$.
\begin{lemma}(Local stability of the grassland periodic solution
$E_G$)\label{stabiliteEG}\\
Suppose that the grassland periodic solution ($E_G$) exists i.e.\\
$\mathcal{R}_G^0>1$ and $\rho_G^0 >1.$\\
Moreover,
\begin{itemize}
    \item[$\bullet$] if $\mathcal{R}_G^T<1$ then, $E_G$ is locally
asymptotically stable,
    \item[$\bullet$] if $(\mathcal{R}_G^T>1$ and $\rho_T<1)$ then, $E_G$ is locally
asymptotically stable,
    \item[$\bullet$] if $(\mathcal{R}_G^T>1$ and $\rho_T=1)$ then, $E_G$ is locally stable.
    \item[$\bullet$] if ($\mathcal{R}_G^T>1$ and $\rho_T>1$) then, $E_G$ is
    unstable.
\end{itemize}
\end{lemma}

\begin{proof}(See Appendix A.)
\end{proof}

\begin{remark}[\textbf{Thresholds interpretation}]
In the sequel, we provide approximative thresholds interpretation in
order to favor an intuitive ecological comprehension of our results
with respect to these thresholds.
\begin{itemize}
    \item[$\bullet$] $\mathcal{R}_G^T$ is the sum of the average amount of biomass
produced by a sensitive/young plant competing with grass between two
successive fires, and the average amount of biomass produced by a
mature plant multiplied by the proportion of young plants which
reach the mature stage.
    \item[$\bullet$] $\rho_T$ embodies both the residual of the reduction of trees
    biomass due to periodic fires events and the reduction of sensitive tree
    biomass due to competition with grass biomass.\\
    Moreover, since $\mathcal{R}<\mathcal{R}_G^T$, after a direct
    computation one has
    \begin{equation}\label{SG}
        \begin{array}{ccc}
          \sigma_G<\displaystyle\frac{1}{G_{int}}\left(\gamma_S-(\mu_S+\omega_S+\mu_{NS})\right) & \Longrightarrow
          & \rho_T >1.
        \end{array}
    \end{equation}
Therefore, it clearly appears following relation (\ref{SG}) that the
grass vs. sensitive tree competition parameter $\sigma_G$ is a
bifurcation parameter for the IFAC model that embodies the
stability/instability of the grassland periodic solution.
\end{itemize}
\end{remark}

\subsubsection{The positive forest equilibrium: existence and local
stability\\} Let
\begin{equation}
\mathcal{R}_T^0=\displaystyle\frac{\gamma_S\mu_{NS}+\gamma_{NS}\omega_S}{\mu_{NS}(\mu_S+\omega_S)}.
\label{seuilR10}
\end{equation}
The following result follows from Proposition 1 in Yatat et al.
(2014) \cite{Yatat2014}.
\begin{lemma}
If $\mathcal{R}_T^0>1$ then system
 $(\ref{fa})-(\ref{ifa})$ has a positive forest equilibrium $E_T=(\bar{T}_S; \bar{T}_{NS}; 0)$,
 where
 \begin{equation}\label{ts}
        \begin{array}{ccc}
          \bar{T}_{S}&=&\displaystyle\frac{K_T\mu_{NS}}{\mu_{NS}+\omega_S}\left(1-\displaystyle\frac{1}{\mathcal{R}_T^0}\right),\\
          \bar{T}_{NS}&=&\displaystyle\frac{\omega_S}{\mu_{NS}}\bar{T}_{S}=\displaystyle\frac{K_T\omega_S}{\mu_{NS}+\omega_S}\left(1-\displaystyle\frac{1}{\mathcal{R}_T^0}\right).
        \end{array}
    \end{equation}
\end{lemma}
As previously we are checking for local stability of the positive
forest equilibrium. Let
\begin{equation}\label{seuilR1G}
    \begin{array}{ccl}
      \mathcal{R}_T^G &=& \displaystyle\frac{\gamma_G}{\mu_G+\sigma_{NS}\bar{T}_{NS}}, \\
      \rho_T^G &=& (1-\eta_G)\exp\left(\gamma_G\left(1-\displaystyle\frac{1}{\mathcal{R}_T^G}\right)\tau\right).
    \end{array}
\end{equation}
Using the same approach as in the proof of Lemma \ref{stabiliteEG},
we derive the following result.
\begin{lemma}(Local stability of the forest equilibrium $E_T$)\label{stabiliteET}\\
Suppose that the forest equilibrium ($E_T$) exists, i.e.
$\mathcal{R}_T^0>1$.\\
Moreover,
\begin{itemize}
    \item[$\bullet$] if $\mathcal{R}_T^G\leq1$ then, $E_T^G$ is locally asymptotically stable,
    \item[$\bullet$] if $(\mathcal{R}_T^G>1$ and $\rho_T^G<1)$ then, $E_T$ is locally asymptotically stable,
    \item[$\bullet$] if $(\mathcal{R}_T^G>1$ and $\rho_T^G=1)$ then, $E_T$ is locally stable .
    \item[$\bullet$] if $(\mathcal{R}_T^G>1$ and $\rho_T^G>1)$ then, $E_T$ is unstable.
\end{itemize}
\end{lemma}

\begin{remark}[\textbf{Thresholds interpretation}]
As we mentioned before, we provide approximative thresholds
interpretation in order to favor an intuitive ecological
comprehension of our results with respect to these thresholds.
\begin{itemize}
    \item[$\bullet$] $\mathcal{R}_T^0$ is the sum of the average amount of biomass produced by a sensitive/young
plant, without fires and competition from grass, and the average
amount of biomass produced by a mature plant multiplied by the
proportion of young plants which reach the mature stage.\\ We may
note here that this threshold only depends on the parameters ruling
the dynamics of the woody biomass.
    \item[$\bullet$] $\mathcal{R}_T^G$ is the average biomass produced by a unit
of grass biomass during its whole lifespan free of fires while
experiencing competition from non-sensitive trees.
    \item[$\bullet$] $\rho_T^G$ embodies both the residual grass biomass after a fire event and the residual biomass of
    the depression of grass biomass due to competition from non-sensitive
    trees.\\
    One should also note that, when $\mathcal{R}_T^G>1$, one has
\begin{equation}\label{SF}
  \begin{array}{ccl}
  \rho_T^G\leq1 & \Longleftrightarrow  & \eta_G\geq1-\displaystyle\frac{1}{\exp\left(\gamma_G\left(1-\displaystyle\frac{1}{\mathcal{R}_T^G}\right)\tau\right)}  \\
  &&\\
  & \Longleftrightarrow & \tau \leq
  -\displaystyle\frac{\ln(1-\eta_G)}{\gamma_G\left(1-\displaystyle\frac{1}{\mathcal{R}_T^G}\right)}\\
  &&\\
  & \Longleftrightarrow &
  \displaystyle\frac{1}{\overline{T}_{NS}}\left(\gamma_G-\mu_G+\displaystyle\frac{\ln(1-\eta_G)}{\tau}\right)\leq\sigma_{NS}.
\end{array}
\end{equation}
    Therefore we can deduce three major observations
    \begin{itemize}
        \item[(i)] Firstly, if the fires period $\tau$ is small (i.e. the fires frequency $f$ is
        high) then
        $1-\displaystyle\frac{1}{\exp\left(\gamma_G\left(1-\displaystyle\frac{1}{\mathcal{R}_T^G}\right)\tau\right)}$
        is small and one can have $\rho_T^G\leq1$ for small values of
        $\eta_G$.
        \item[(ii)] Secondly, if the fires period $\tau$ is large (i.e. the fires frequency $f$ is
        small) then
        $$1-\displaystyle\frac{1}{\exp\left(\gamma_G\left(1-\displaystyle\frac{1}{\mathcal{R}_T^G}\right)\tau\right)}=1-\varepsilon,
        $$
        where $0<\varepsilon<<1$. Thus, $\varepsilon$ can be sufficiently small (such that
        $\rho_T^G\leq1$), to have a large destruction of grass biomass (more than 99\%). This is ecologically not possible  (according to the fact that a part of the grass biomass like roots and even the bottom of the tufts, cannot
        burn). Therefore, in case of large fires period, having
        $\rho_T^G\leq1$ may likely correspond to a decrease of $\mathcal{R}_T^G$ (i.e. an increase of $\sigma_{NS}$).
        \item[(iii)] It is easily deduced from relation (\ref{SF})
        that the non sensitive tree vs. grass competition/facilitation parameter
        $\sigma_{NS}$ and the fire return time $\tau$ are bifurcation parameters for the IFAC model that
        embody  stability/instability of the forest equilibrium.
    \end{itemize}
\end{itemize}
\end{remark}
\begin{remark}
A direct comparison leads to:
\begin{enumerate}
    \item $\mathcal{R}_G^T<\mathcal{R}_T^0$
    \item $\mathcal{R}_G^0<1\Longrightarrow \mathcal{R}_T^G<1,$
    \quad $\mu_G>0.$
\end{enumerate}
\end{remark}

\subsubsection{Global stability of trivial equilibrium (desert) and
semi-trivial solutions (grassland periodic solution and the forest
equilibrium)\\} Here, we state a result concerning the global
stability of the desert equilibrium, the forest equilibrium and a
result concerning the global stability of the grassland periodic
solution. Using the thresholds defined in (\ref{seuilR20}),
(\ref{seuilR20tauetaG}) and (\ref{seuilR10}), we have the following
\begin{theorem}\label{prop}
$\bullet$ Case 1: $\mu_G>0.$
\begin{enumerate}
    \item If $\mathcal{R}_T^0<1$ and $\mathcal{R}_G^0<1$ then the
    desert equilibrium $E_0$ is globally asymptotically stable (GAS).
    \item If $\mathcal{R}_T^0>1$ and $\mathcal{R}_G^0<1$ then the
    forest  equilibrium $E_T=(\bar{T}_S,\bar{T}_{NS},0)$, where $(\bar{T}_S,\bar{T}_{NS})$ are given
    in (\ref{ts}), is globally asymptotically stable.
    \item If $\mathcal{R}_T^0<1$, $\mathcal{R}_G^0>1$ and
    $\rho_G^0<1$
    then the
    desert equilibrium $E_0$ is globally asymptotically stable.
    \item If $\mathcal{R}_T^0<1$, $\mathcal{R}_G^0>1$ and $\rho_G^0>1$
    then the
    grassland periodic solution $E_G=(0,0,G^*(t))$, where $G^*(t)$ is given by (\ref{getoile}), is globally asymptotically stable.
    \end{enumerate}
\qquad  $\bullet$  Case 2: $\mu_G=0.$
\begin{itemize}
    \item[(i)]If $\mathcal{R}_T^0<1$ and $\rho_G^0<1$ then the
    desert equilibrium $E_0$ is globally asymptotically stable.
    \item[(ii)]If $\mathcal{R}_T^0<1$ and $\rho_G^0>1$
    then the
    grassland periodic solution $E_G=(0,0,G^*(t))$, where $G^*(t)$ is given by (\ref{getoile}), is globally asymptotically stable.
    \item[(iii)]If $\mathcal{R}_T^0>1$ and $\rho_G^0<1$ then the
    forest  equilibrium $E_T=(\bar{T}_S,\bar{T}_{NS},0)$, where $(\bar{T}_S,\bar{T}_{NS})$ are given
    in (\ref{ts}), is globally asymptotically stable.
    \end{itemize}
\end{theorem}

\begin{proof} (See Appendix B.)
\end{proof}

\subsection{ Existence of a positive and periodic tree-grass
solution} Now, we reach the position to find at least one
non-trivial positive and periodic solution of system
(\ref{fa})-(\ref{ifa}). We will use
the approach developed by Gaines and Mahwin 1977 \cite{Gaines1977}.\\
Before we give the main result of this section, we recall useful inequalities according to the thresholds defined in (\ref{seuilR20}) and (\ref{seuilR20tauetaG})
\begin{enumerate}
    \item If $\mathcal{R}_G^0<1$ then
    $\rho_G^0<1$.
    \item $\rho_G^0>1 \Longleftrightarrow (\gamma_G-\mu_G)+\displaystyle\frac{\ln(1-\eta_G)}{\tau}>0.$
\end{enumerate}

The following result holds for system $(\ref{fa})-(\ref{ifa}).$
\begin{theorem}(\textbf{Existence of a positive and periodic savanna
solution})\\\label{sol_coexist} $\bullet$ Case 1: $\mu_G>0$.\\ If
$\mathcal{R}_G^0>1$ and $\rho_G^0>1$ then system
$(\ref{fa})-(\ref{ifa})$ has at least one positive $\tau$-periodic
solution.\\
$\bullet$ Case 2: $\mu_G=0$.\\
If $\rho_G^0>1$  then system $(\ref{fa})-(\ref{ifa})$ has at least
one positive $\tau$-periodic solution.
\end{theorem}

\begin{proof}(See Appendix C.)
\end{proof}

\begin{remark}One should note that Theorem \ref{sol_coexist}
provides only sufficient conditions to ensure existence of at least
one positive $\tau$-periodic solution of system
(\ref{fa})-(\ref{ifa}). In other words, existence of savanna
solution relies on sufficient grass biomass production. Thus, from
an ecological point of view assumptions of Theorem \ref{sol_coexist}
can also be view as necessary conditions to have savanna solution.
\end{remark}

\begin{remark}
Note also, that the uniqueness of the positive savanna solution is
an open problem. In the rest of the paper, we will assume that we
only have one positive solution.
\end{remark}

As previously we can check the local asymptotic stability of the positive and periodic savanna solution $E_{TG}=(\tilde{T}_S(t),
\tilde{T}_{NS}(t), \tilde{G}(t))$. Defining
 \begin{equation}
    \begin{array}{lcl}
      T_S(t) & = & x(t)+\tilde{T}_{S}(t), \\
      T_{NS}(t) & = &  y(t)+\tilde{T}_{NS}(t), \\
      G(t) & = &  z(t)+\tilde{G}(t),
    \end{array}
 \end{equation}
where $x(t), y(t)$ and $z(t)$ are small perturbations and satisfy
\begin{equation}
    \left(
      \begin{array}{c}
        x(t) \\
        y(t) \\
        z(t) \\
      \end{array}
    \right)=\Phi(t)\left(
                     \begin{array}{c}
                       x(0) \\
                       y(0) \\
                       z(0) \\
                     \end{array}
                   \right),
\end{equation}
where $\Phi$ is a fundamental matrix and satisfies
\begin{equation}
    \displaystyle\frac{d\Phi(t)}{dt}=DF(\tilde{T}_{S}(t); \tilde{T}_{NS}(t);
    \tilde{G}(t))\Phi(t)=\left(
                \begin{array}{ccc}
                  a^{(1)} & a^{(2)} & -\sigma_G\tilde{T}_S(t) \\
                  \omega_S & -\mu_{NS} & 0 \\
                  0 & -\sigma_{NS}\tilde{G}(t) & a^{(3)} \\
                \end{array}
              \right)\Phi(t)
\end{equation}
with
\begin{equation}
    \begin{array}{ccl}
      a^{(1)} & = & \gamma_S\left(1-\displaystyle\frac{2\tilde{T}_S(t)+\tilde{T}_{NS}(t)}{K_T}\right)-\displaystyle\frac{\gamma_{NS}}{K_T}\tilde{T}_{NS}(t)-\sigma_G\tilde{G}(t)-\mu_S-\omega_S, \\
      a^{(2)} & = &
      \gamma_{NS}\left(1-\displaystyle\frac{\tilde{T}_S(t)+2\tilde{T}_{NS}(t)}{K_T}\right)-\displaystyle\frac{\gamma_{S}}{K_T}\tilde{T}_{S}(t),\\
      a^{(3)} & = &
      \gamma_G\left(1-\displaystyle\frac{2\tilde{G}(t)}{K_G}\right)-\sigma_{NS}\tilde{T}_{NS}(t)-\mu_G
    \end{array}
\end{equation}
 and
$\Phi(0)=Id_{\mathbb{R}^3}$. Furthermore, the resetting impulsive
condition of system (\ref{fa})-(\ref{ifa}) becomes,
\begin{equation}
                        \left(
                          \begin{array}{c}
                            x(n\tau^+) \\
                            y(n\tau^+) \\
                            z(n\tau^+) \\
                          \end{array}
                        \right)
                        =\left(
                           \begin{array}{ccc}
                             1-\eta_Sw(\tilde{G}(\tau)) & 0 & -w'(\tilde{G}(\tau))\tilde{T}_S(\tau) \\
                             0 & 1 & 0 \\
                             0 & 0 & 1-\eta_G \\
                           \end{array}
                         \right)\left(
                          \begin{array}{c}
                            x(n\tau) \\
                            y(n\tau) \\
                            z(n\tau) \\
                          \end{array}
                        \right).
\end{equation}
A monodromy matrix $\mathbf{M}$ of system $(\ref{fa})-(\ref{ifa})$,
is:
\begin{equation}\label{mt}
    \mathbf{M}=\left(
                           \begin{array}{ccc}
                             1-\eta_Sw(\tilde{G}(\tau)) & 0 & -w'(\tilde{G}(\tau))\tilde{T}_S(\tau) \\
                             0 & 1 & 0 \\
                             0 & 0 & 1-\eta_G \\
                           \end{array}
                         \right)\Phi(\tau),
\end{equation}
with
\begin{equation}\label{phit}
    \begin{array}{ccc}
      \Phi(t) & = & \exp\left(\displaystyle\int_0^tDF(\tilde{T}_{S}(s); \tilde{T}_{NS}(s);
    \tilde{G}(s))ds\right).
    \end{array}
\end{equation}
Let defined $\rho_{TG}$ such as
\begin{equation}\label{rhoTGGG}
\rho_{TG}=\max(|\lambda| ~:~\lambda\in sp(\mathbf{M})),
\end{equation}
where the matrix $\mathbf{M}$ is defined in $(\ref{mt})$. Therefore,
following the Floquet theorem (D'Onofrio (2002) \cite{Donofrio2002},
Chen et al. (2009) \cite{Chen2009})  we deduce the following results
\begin{lemma}
\begin{itemize}
    \item[$\bullet$] If $\rho_{TG}<1$ then the positive $\tau$-periodic
    solution of system $(\ref{fa})-(\ref{ifa})$ is locally
    asymptotically stable.
    \item[$\bullet$] If $\rho_{TG}=1$ then the positive $\tau$-periodic
    solution of system $(\ref{fa})-(\ref{ifa})$ is locally stable.
    \item[$\bullet$]If $\rho_{TG}>1$ then the positive $\tau$-periodic
    solution of system $(\ref{fa})-(\ref{ifa})$ is unstable.
\end{itemize}
\end{lemma}
Unfortunately, expressions $(\ref{mt})$ and $(\ref{phit})$ don't
allow an explicit computation of the eigenvalues of the monodromy
matrix $\mathbf{M}$ and of the real $\rho_{TG}$. Therefore, the
stability of the positive $\tau$-periodic solution of system
$(\ref{fa})-(\ref{ifa})$ will be conjectured through numerical
computation of the threshold $\rho_{TG}$.

\subsection{Summary Table of the IFAC model qualitative analysis}
 Based on the previous studies, we deduce the summary Table
of the qualitative analysis of the IFAC model. In Table
\ref{recapitulatif},
\textbf{L} stands for Locally Asymptotically Stable,
\textbf{NU} stands for Numerical Asymptotical Stability and the
empty cell denotes either the instability/non existence of the
corresponding solution or that the result does not depend on the
corresponding threshold. For reader's convenience, we recall all
thresholds defined previously:
\begin{equation}
    \begin{array}{lcl}
       \mathcal{R}_T^0 &=& \displaystyle\frac{\gamma_S\mu_{NS}+\gamma_{NS}\omega_S}{\mu_{NS}(\mu_S+\omega_S)}, \\
       &&\\
       \mathcal{R}_G^0 &=& \displaystyle\frac{\gamma_G}{\mu_G}, \qquad \mbox{when }\mu_G > 0\\
       &&\\
              \rho_G^0 &=&(1-\eta_G)\exp\left(\mu_G \left(\mathcal{R}_G^0-1\right)\tau\right), \\
              &&\\
       \mathcal{R}_T^G &=& \displaystyle\frac{\gamma_G}{\mu_G+\sigma_{NS}\bar{T}_{NS}}, \qquad \mbox{where } \bar{T}_{NS}= \displaystyle\frac{K_T\omega_S}{\mu_{NS}+\omega_S}\left(1-\displaystyle\frac{1}{\mathcal{R}_T^0}\right)\\
       &&\\
       \rho_T^G &=& (1-\eta_G)\exp\left(\gamma_G\left(1-\displaystyle\frac{1}{\mathcal{R}_T^G}\right)\tau\right),  \\
       &&\\
       \mathcal{R}_G^T & = &\displaystyle\frac{\gamma_S\mu_{NS}+\omega_S\gamma_{NS}}{\mu_{NS}(\mu_S+\omega_S)+\displaystyle\frac{\mu_{NS}\sigma_G}{\tau}\displaystyle\frac{K_G}{\gamma_G}\left(\ln(1-\eta_G)+(\gamma_G-\mu_G)\tau\right)}
    \end{array}
\end{equation}
and $\rho_T$ is defined in (\ref{rhot}).

In Table \ref{recapitulatif}, we implicitly assume that $R_G^0>1$
and $R_T^0>1$.
\begin{table}[H]
\label{recapitulatif} \centering
\vspace{0.25cm}
\begin{tabular}{|l|l|l|l|l|c|c|c|c|}
  \hline
\multicolumn{5}{|c|}{Thresholds}&\multicolumn{3}{|c|}{Solutions}& \\
\cline{1-8}
&&&&&&&&\\
 $\mathcal{R}_T^G$ & $\rho_G^0$ & $\rho_T^G$ & $\mathcal{R}_G^T$ & $\rho_T$   &  $E_T$& $E_G$& $E_{TG}$ & Case \\
&&&&&&&&\\
  \hline
   $>1$ & $>1$ & $>1$ & $>1$ & $>1$ &  &  & NU & 1\\
\cline{5-9}
 &  &  &  & $\leq1$ &  & \textbf{L}  & NU & 2\\
\cline{4-9}
 &  &  & $<1$ &  &  & \textbf{L} & NU & 3\\
 \cline{3-9}
 &  & $\leq1$ & $>1$ & $>1$ & \textbf{L} &  & NU &4\\
\cline{5-9}
 &  & &  & $\leq1$ & \textbf{L} & \textbf{L} & NU & 5\\
\cline{4-9}
  &  &  & $<1$ &  &  \textbf{L} & \textbf{L} & NU & 6\\
\cline{2-9}
 & $<1$ &  &  &  & \textbf{L} &  & NU & 7\\
\cline{1-9}
 $<1$ & $>1$ &  & $>1$ & $>1$ &  \textbf{L} &  & NU & 8\\
\cline{5-9}
  &  &  &  & $\leq1$  & \textbf{L} & \textbf{L} & NU &9\\
\cline{4-9}
  &  &  & $<1$ &  & \textbf{L} & \textbf{L} & NU & 10\\
\cline{2-9}
 & $<1$ &  &  &  & \textbf{L} &  & NU & 11\\
  \hline
\end{tabular}
\caption{Summary table of the qualitative analysis of system
$(\ref{fa})-(\ref{ifa})$. \textbf{NU}: LAS for $E_{TG}$ needs to be estimated numerically, using }

\end{table}

\begin{remark}
In Table \ref{recapitulatif}, we only summarizes the cases where
$\mathcal{R}_G^0>1$ and $\mathcal{R}_T^0>1$, which are the
interesting cases from the ecological point of view. Moreover, a
direct computation leads
\begin{equation}\label{comparaison_de_rho}
\begin{array}{ccl}
  \rho_T^G &=& (1-\eta_G)\exp\left(\gamma_G\left(1-\displaystyle\frac{1}{\mathcal{R}_T^G}\right)\tau\right) \\
  &&\\
   & = & (1-\eta_G)\exp\left((\gamma_G-\mu_G-\sigma_{NS}\bar{T}_{NS})\tau\right) \\
   &&\\
   & = & (1-\eta_G)\exp\left(\mu_G(\mathcal{R}_G^0-1)\tau\right)\exp\left(-\sigma_{NS}\bar{T}_{NS}\tau\right) \\
   &&\\
  \rho_T^G & = & \rho_G^0\exp(-\tau\sigma_{NS}\bar{T}_{NS}).
\end{array}
\end{equation}
Therefore, according to (\ref{comparaison_de_rho}), since
$\mathcal{R}_T^0>1$ i.e. $\bar{T}_{NS}>0$ then $\rho_T^G<\rho_G^0$.
\end{remark}

In the sequel, we provide some numerical simulations in order to
illustrate our theoretical results. To achieve that goal, first we
will provide a suitable nonstandard numerical scheme which will be
helpful for the numerical approximation of the IFAC model solutions.
It is well know that standard methods (such as Runge-Kutta or Euler
methods) can sometimes present spurious behaviors which are not in
adequacy with the system properties that they aim to approximate
i.e., lead to negative solutions, exhibit numerical instabilities,
or even converge to the wrong equilibrium for certain values of the
time discretization or the model parameters (interested readers can
also see Yatat et
 al. (2014) \cite{Yatat2014}, Anguelov et al. (2012)
\cite{Anguelov2012}, (2013) \cite{Anguelov2013}, (2014)
\cite{Anguelov2014} and Dumont et al. (2010) \cite{Dumont2010},
(2012) \cite{Dumont2012} for motivations, details and explanations
about nonstandard schemes). Secondly, we will focus on three
ecological regions of the African continent that contrast in terms
of biomass production conditions, namely a semiarid, a mesic and a
humid tropical region, to discuss the IFAC outcomes with respect to
published modelling results on savanna ecosystems (Baudena et al.
(2014) \cite{Baudena2014}, February et al. (2013)
\cite{February2013}, Accatino et al. (2010) \cite{Accatino2010},
Mordelet et al. (1995) \cite{Mordelet1995}, Moustakas et al. (2013)
\cite{Moustakas2013}).


\section{Nonstandard scheme, parameters ranges and ecological zones of biomass productions}\label{NS}
\subsection{A nonstandard scheme for the IFAC model} The
nonstandard numerical scheme proposed in this section is adapted
 from the nonstandard scheme proposed for the COFAC model in Yatat et
 al. (2014) \cite{Yatat2014}.

  System (\ref{fa}) is discretized as
 follows:
 \begin{equation}\label{ds}
 \left\{%
\begin{array}{rcl}
\displaystyle\frac{G^{k+1}-G^k}{\phi_G(h)} &=&
 \gamma_G\left(1-\displaystyle\frac{G^{k+1}}{K_G}\right)G^{k}-\sigma_{NS}T_{NS}^{k}G^{k+1}-\mu_GG^{k},\\
 & &\\
\displaystyle\frac{T_{NS}^{k+1}-T_{NS}^k}{\phi(h)} &=& \omega_ST_S^{k}-\mu_{NS}T_{NS}^{k+1}, \\
& &\\
\displaystyle\frac{T_S^{k+1}-T_S^k}{\phi(h)}&=&(\gamma_S-(\mu_S+\omega_S))T_{S}^{k}+\gamma_{NS}T_{NS}^{k+1}
-\displaystyle\frac{\gamma_{S}}{K_T}T_{S}^{k}(T_S^{k+1}+T_{NS}^{k+1})\\
&&\\
  & &-\displaystyle\frac{\gamma_{NS}}{K_T}T_{NS}^kT_{NS}^{k+1}-\left(\displaystyle\frac{\gamma_{NS}}{K_T}T_{NS}^k+\sigma_G
   G^k\right)T_{S}^{k+1},
  \end{array}
\right.
\end{equation}
and the impulsive event (\ref{ifa}) is discretized as
 follows:

\begin{equation}\label{ids}
 \left\{%
\begin{array}{rcl}
G^{k+1}_+& = &(1-\eta_G)G^{k+1}\\
&&\\
T_{NS+}^{k+1} &=& T_{NS}^{k+1}, \\
&&\\
T_{S+}^{k+1} &=& (1-\eta_Sw(G^{k+1}))T_{S}^{k+1},
  \end{array}
\right.
\end{equation}
where the denominator functions $\phi$ and $\phi_1$ read as
\begin{equation}\label{q}
\phi(h)=\frac{e^{Qh}-1}{Q},~~h>0,
\end{equation}
with
\begin{equation}\label{q1}
Q= \max\left(\mu_{NS}, \gamma_S-(\mu_S+\omega_S)\right).
\end{equation}
Using the fact that $\gamma_G-\mu_G=\mu_G(\mathcal{R}_G^0-1)$, we
define
\begin{equation}\label{q2}
\phi_G(h)=\left\{
\begin{array}{l}
 \displaystyle\frac{e^{\mu_G(\mathcal{R}_G^0-1)h}-1}{\mu_G(\mathcal{R}_G^0-1)},~~\mu_G>0.\\
 \\
 \displaystyle\frac{e^{\gamma_Gh}}{\gamma_G},~~\mu_G=0,~~h>0.
 \end{array}
 \right.
\end{equation}
This scheme is positively invariant and is qualitatively stable,
which means that it has the same equilibria than system (\ref{fa})-(\ref{ifa}),
and the stability/instability properties of the equilibria are
preserved, at least locally, whatever the stepsize $h>0$, \cite{Yatat2014}.

\subsection{Parameters ranges and ecological zones of biomass productions}
To provide relevant numerical simulations, one need to use
ecologically meaningful parameters ranges and values. Thus, after
extensive literature review, we found the following parameters
ranges:
\begin{table}[H]
  \centering
  \caption{Parameters values found in literature}
  \vspace{0.5cm}
  \begin{tabular}{ccl}
    \hline
    Parameters & values & References \\
    \hline
    $f ~(1/\tau)$ & 0 -- 1 & Langevelde et al. (2003) \cite{Langevelde2003} \\
        & 0 -- 2 & Accatino et al. (2010) \cite{Accatino2010} \\
    $\gamma_G$ & $0.4^{(1)}$ -- $4.6^{(2)}$ & $^{(1)}$ Penning de Vries (1982) \cite{Penning1982} \\
    &&$^{(2)}$ Menaut et al. (1979) \cite{Menaut1979}\\
    $\gamma_S+\gamma_{NS}$ & 0.456 -- 7.2 & Breman et al. (1995) \cite{Breman1995}\\
    $\mu_{NS}$ & 0.03 -- 0.3 & Accatino et al. (2010) \cite{Accatino2010}\\
    & 0.4& Langevelde et al. (2003) \cite{Langevelde2003}\\
    $\mu_S$  & 0 -- 0.3 & Langevelde et al. (2003) \cite{Langevelde2003} \\
    $\mu_G$  & 0 -- 0.6 & Langevelde et al. (2003) \cite{Langevelde2003} \\
    $\eta_G$ & 0.1$^{(a)}$ -- 1$^{(b)}$ & $^{(a)}$ Van de Vijver (1999) \cite{Vanvijer1999}\\
    && $^{(b)}$ Accatino et al. (2010) \cite{Accatino2010}\\
    &0.2 -- 1 & Abbadie et al. (2006) \cite{Abbadie2006} \\
     $\eta_S$ & 0.02 -- 0.6 &  Accatino et al. (2010) \cite{Accatino2010}\\
     & 0.66 & Reinterpretation of Gignoux (1994) \cite{Gignoux1994}, \\
     &&Reinterpretation of Langevelde et al. (2003) \cite{Langevelde2003} \\
     $\omega_S$ & 0.05 -- 0.2 & Wakeling et al. (2011)
     \cite{Wakeling2011}\\
     \hline
  \end{tabular}
\end{table}

As stated previously, we will focus our numerical simulations on three ecological zones of the African continent
\begin{itemize}
\item[Region 1] is a semiarid area  with mean annual rainfall comprised between 300 $mm. yr^{-1}$ and 650 $mm. yr^{-1}$ where there is a low biomass production and few fire occurrence (says one fire event every ten years) if any.
\item[Region 2] is a mesic area with a mean annual rainfall that varies between 650 $mm. yr^{-1}$
 and 1100 $mm. yr^{-1}$ which is an intermediate biomass production zone  and
where we can have on average one fire event every four or five years and sometimes less.
\item[Region 3] is a high biomass production
zone, in which we can have one or two fire events per year, i.e. a humid tropical area with a mean annual rainfall between 1100 $mm. yr^{-1}$ and 1800 $mm. yr^{-1}$.
\end{itemize}
Our aim is to assess the different outcomes of the IFAC model along
with the influence of the variations of $\sigma_G$, $\sigma_{NS}$
(for which there is no direct information in the published
literature) and the fire period $\tau$.
 The
parameter ranges in each of the regions are summarized in Table
\ref{RR}.
\begin{table}[H]
  \centering
  \caption{Parameters ranges in the three ecological regions}\label{RR}
  \vspace{0.5cm}
  \begin{tabular}{l|c|c|c}
    \hline
    Parameter & Region 1 & Region 2 & Region 3 \\
    \hline
    $\tau$ ($yr$) & $>$ 5 & 2 -- 5 & 0.5 -- 3 \\
    $K_T$ $(t.ha^{-1})$ & 30 & 80 -- 90 & 110 -- 120 \\
    $K_G$ $(t.ha^{-1})$& 0 -- 5 & 5 -- 10 & 10 -- 20 \\
    $\gamma_G$ $(yr^{-1})$& 0.4 -- 2 & 2 -- 3.5 & 3.5 -- 4.6 \\
    $\gamma_S$ $(yr^{-1})$&  0.2 -- 0.8 & 0.2 -- 1 & 1.5 -- 2.7 \\
    $\gamma_{NS}$ $(yr^{-1})$& 0.256 -- 1.2 & 1.2 -- 2.5 & 2.5 -- 4.5 \\
    $\mu_{NS}$ $(yr^{-1})$ & 0.1 -- 0.25 & 0.07 -- 0.1 & 0.02 -- 0.07\\
    \hline
  \end{tabular}
\end{table}
In addition, reinterpreting experiments that concern Region 1 and
Region 2 and reported in February et al. (2013) \cite{February2013}
we derived $\sigma_{G}$ (in $ha.t^{-1}.yr^{-1}$): 0.1843 -- 0.9984
for Region 1 and $\sigma_{G}$: 0.2470 -- 1.6287 for Region 2.
Moreover, several studies located under different rainfall compared
grass production under and outside a tree crown. A synthesis was
proposed by Mordelet \& Le Roux (see Abadie et al. (2006) Page 156
\cite{Abbadie2006}) that emphasized that the relative production
(within to outside) is a decreasing function along the rainfall
gradient. We re-interpreted the results as to derive reasonable
values for $\sigma_{NS}$ in each of the three regions, using the
subsequent reasoning.
\par
Assuming that the measurements were made, free of fires, in grass
stands having reached equilibrium, and letting $G_u$ and $G_o$ be
the equilibrium values under and outside crown, respectively. We can
write according to the model:
\begin{equation}
    \begin{array}{ccl}
      G_u & = & K_G\left(1- \displaystyle\frac{\mu_G+\sigma_{NS}\tilde{T}}{\gamma_G}\right)\\
      &&\\
      G_o & = & K_G\left(1- \displaystyle\frac{\mu_G}{\gamma_G}\right).
    \end{array}
\end{equation}
The ratio considered by Mordelet \& Le Roux (1995)
\cite{Mordelet1995} (see also Abadie et al. (2006) Page 156
\cite{Abbadie2006}) is:
\begin{equation}\label{ratio_deltaG}
  \delta_G=\frac{G_u}{G_o},
\end{equation}
 i.e. the ratio of grass
production under and outside a tree crown. Assuming $\mu_G = 0$ we
can simplify as:
\begin{equation}\label{ratio}
\delta_G=1-\displaystyle\frac{\sigma_{NS}\tilde{T}}{\gamma_G}\Longleftrightarrow
\sigma_{NS}=\displaystyle\frac{(1-\delta_G)\gamma_G}{\tilde{T}}.
\end{equation}
$\tilde{T}$ is the woody biomass density to be computed at the scale
of an isolated, full grown tree (having reached the maximal height
considering the local climate) in any of the three regions.\\
 We propose to relate $\tilde{T}$ to $K_T$ as
$$\tilde{T}=\displaystyle\frac{\varepsilon\times K_T}{S},$$ where $S$
is the woody cover characterizing the maximal density $K_T$ and
$\varepsilon\in]0, 1[$ is a coefficient expressing that an isolated
tree has less influence on grass production that a complete, closed
canopy stand corresponding to $K_T$.\par

We used the value of $\delta_G=1.58$ (resp. 1.25 and 0.75, 0.75) of
Mordelet \& Le Roux (1995) \cite{Mordelet1995} (see also Abadie et
al. (2006) Page 156 \cite{Abbadie2006}) that corresponds to Region 1
(semiarid region) (resp. Region 2, Region 3). Note that $\delta_G$
values above 1 express a facilitative effect of trees for grass
while values below correspond to a depressing effect. Using also the
estimated values for $K_T$ and $\gamma_G$, we deduce the ranges of
variation of $\sigma_{NS}$ (in $ha.t^{-1}.yr^{-1}$), summarized in
Table \ref{table_sigmaNS}.
\begin{table}[H]\label{table_sigmaNS}
  \centering
  \caption{Variation range of $\sigma_{NS}$ in Region 1, Region 2 and Region 3 following re-interpretation of Mordelet \& Le Roux results}
\label{table_sigmaNS}\vspace{0.15cm}
\begin{tabular}{c|c|c|c|c|c|c|c|c|}
  \cline{2-9}
   & $\delta_G$ & $\gamma_G$ & $K_T$ & $\varepsilon~(\min)$ & $\varepsilon~(\max)$ & $S$ & $\sigma_{NS}~(\min)$ & $\sigma_{NS}~(\max)$ \\
   \hline
  Region 1 & 1.58 & 0.6 & 30 & 0.4 & 0.75 & 1 & -0.029 & -0.0155 \\
  \hline
  \hline
  Region 2 & 1.25 & 2.8 & 85 & 0.2 & 0.67 & 1 & -0.0412 & -0.0123 \\
  \cline{2-2}\cline{8-9}
          & 0.75 &  &  &  &  &  & 0.0123 & 0.0412 \\
 \hline
 \hline
  Region 3 & 0.75 & 4.2 & 115 & 0.1 & 0.15 & 1 & 0.0609 & 0.0913 \\
  \hline
\end{tabular}
\end{table}
Note that the range of values for $\sigma_{NS}$ is due to a large
uncertainty on $\varepsilon$. Moreover as a straightforward
consequence of the results of the reference study (Modelet \& Le
Roux (1995) \cite{Mordelet1995}, Abadie et al. (2006)) $\sigma_{NS}$
is likely to be negative in Region 1 (shading improves grass
production in Region 1, which is also in good agreement with results
of Moustakas et al. (2013) \cite{Moustakas2013}, Belsky et al.
(1989) \cite{Belsky1989}, Weltzin \& Coughenour (1990)
\cite{Weltzin1990}).

\section{Numerical simulations and discussion}\label{discussion}
\subsection{Results for Region 1}
In semiarid areas, the main mechanisms that govern the ecological
processes include
\begin{itemize}
    \item[\textbf{(M1)}] water limitation on tree growth (Baudena et al. (2014) \cite{Baudena2014}) and on grass biomass standing crop
    ($K_G$)
    \item[\textbf{(M2)}] tree - grass competition, which has an
    especially strong competitive impact on tree seedlings (February et al. (2013)
    \cite{February2013})
    \item[\textbf{(M3)}] unfrequent fire may reduces woody cover and grass cover but
    grass biomass recover quickly after fire (Baudena et al. (2014) \cite{Baudena2014}) while low values of grass biomass limit the impact of competition on sensitive
    tree biomass.
\end{itemize}
Point \textbf{(M2)} suggests that the sensitive tree vs. grass
competition parameter $\sigma_G$ has relatively large values while
point \textbf{(M1)} along with point \textbf{(M3)} suggest that
woody cover is controlled principally by water availability and
secondarily by unfrequent fires (Sankaran et al. (2005)
\cite{Sankaran2005}). Thus the depressive effect of scattered woody
cover on grass biomass mainly results from reduced light
availability and root competition for soil water (see Walker et al.
(1981) \cite{Walker1981}). Nevertheless, somme references (Belsky et
al. (1989) \cite{Belsky1989}, Weltzin \& Coughenour (1990)
\cite{Weltzin1990}) also emphasized the facilitation role of
scattered tree on grass biomass in East African semiarid savannas.
Indeed, compared with the open situation, the highest grass
production was recorded under acacia and baobab trees (Mordelet \&
Menaut (1995) \cite{Mordelet1995}, Belsky et al. (1989)
\cite{Belsky1989}, Weltzin \& Coughenour (1990) \cite{Weltzin1990}),
which are known to have a low light interception and only induce a
slight limitation to photosynthesis while shading improves the water
balance under the canopy (Barbier et al. 2008) \cite{Barbier2008}.
This finding of Mordelet \& Menaut (1995) \cite{Mordelet1995},
Belsky et al. (1989) \cite{Belsky1989}, Weltzin \& Coughenour (1990)
\cite{Weltzin1990} can also be explained by the soil enrichment by
nitrogen fixing species, like acacias trees which results in a yield
increase. Therefore, in semiarid areas the main ecological
vegetation types that are observed depending on annual rainfall
(which only varies tree/grass ratio) are savannas (February et al.
(2013) \cite{February2013}, Baudena et al. (2014)
\cite{Baudena2014}, Accatino et al. (2010) \cite{Accatino2010}) and
sometimes forest in the sense of low dry forest, and/or thickets
(Walker et al. (1981) \cite{Walker1981} , Couteron \& Kokou (1997)
\cite{Couteron1997}). The outcome of the IFAC model in semiarid
areas is also in adequacy with this previous features. Indeed,  let
us consider the following table of parameters values
  \begin{table}[H]
  \centering
  \caption{Array of parameters' values for Region 1}\label{R1}
  \vspace{0.25cm}
  \small
  \begin{tabular}{|c|c|c|c|c|c|c|}
    \hline
 $\gamma_S$ & $\gamma_{NS}$ & $\gamma_G$ & $\mu_{NS}$ &  $\omega_S$ & $\mu_S$ & $\eta_S$\\
     \hline\hline
0.3 & 1 & 0.6 & 0.15 & 0.1 & 0.2 & 0.5\\
    \hline\hline
       \multicolumn{7}{|c|}{ $\tau$=7,~~ $K_T$=30,~~ $K_G=2.5$}\\
    \hline
  \end{tabular}
\end{table}

For values in Table \ref{R1}, we compute
\begin{center}
\begin{tabular}{|c|c|}
  \hline
 $\mathcal{R}_T^0$  &  $\mathcal{R}_G^0$\\
    \hline
  $3.2222$  & $2$\\
  \hline
\end{tabular}
\end{center}
and we derive figure \ref{R1fig1}.
\begin{figure}[H]
    \centering
    \resizebox{1\textwidth}{!}{\includegraphics{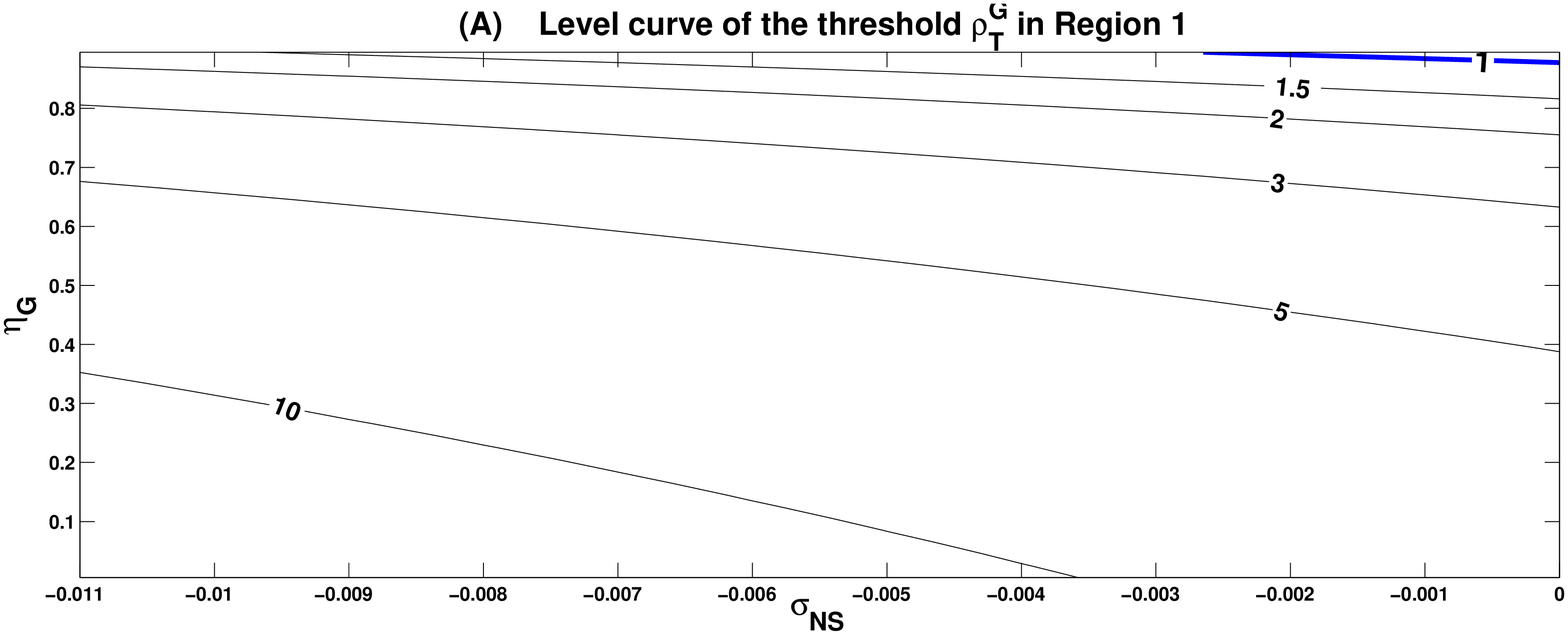}}
    \resizebox{1\textwidth}{!}{\includegraphics{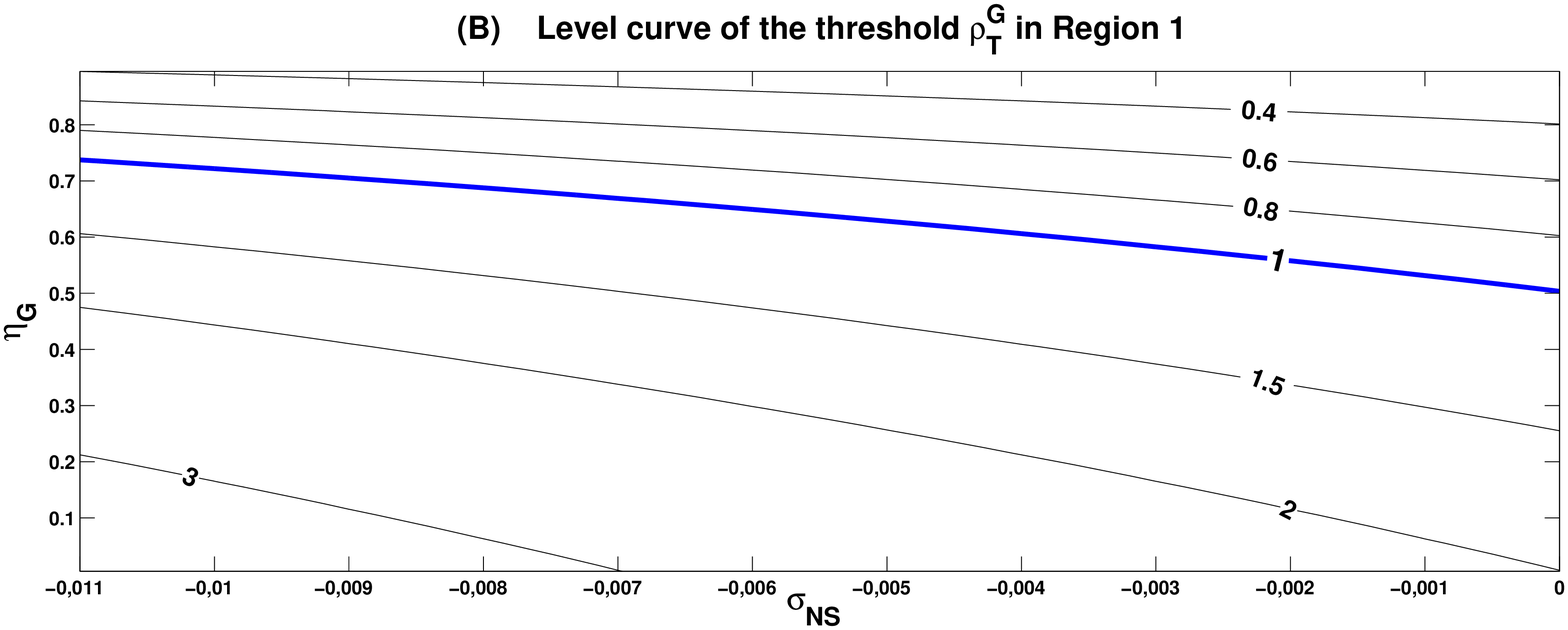}}
 \caption{Level curve of the threshold $\rho_T^G$ illustrating that as $\eta_G$, $\mu_G$ and $\sigma_{NS}$ increase, $\rho_T^G$ decreases and
 system (\ref{fa})-(\ref{ifa}) is liable to move
 from a savanna/grassland state to a forest state or to a multistability involving the forest solution.
  Recall that the forest solution is stable (resp. unstable) whenever $\rho_T^G$ is lower (resp. greater) than unity. In (A) $\mu_G=0.3$, in (B) $\mu_G$=0.5.} \label{R1fig1}
 \end{figure}

Recall that according to relation (\ref{seuil_etaG}),
$$\rho_G^0>1\Longleftrightarrow \eta_G< \left\{\begin{array}{cl}
                                                 0.8775& \mbox{for $\mu_G=0.3$} \\
                                                 0.5034& \mbox{for
                                                 $\mu_G=0.5$}.
                                               \end{array}\right.
 $$
 Moreover, together with our data estimation of
$\sigma_G$ and $\sigma_{NS}$ in Region 1, we found that
$\mathcal{R}_T^G>1$. Setting $\mu_G=0.3$, $\eta_G=0.6$, and
\begin{itemize}
    \item when $\sigma_G\in [0.92, 0.95[$, we have $\mathcal{R}_G^T>1$ and $\rho_T\leq1$
    \item when $\sigma_G\in [0.95, 0.9984[$, we have $\mathcal{R}_G^T\leq1$ and
    $\rho_T\leq1$.
\end{itemize}
 Finally, when $\mu_G=0.5$ one gets $\rho_T>1$.
Therefore, figure \ref{R1fig1} together with the previous discussion
on $\mathcal{R}_G^T$ and $\rho_T$ illustrates either case 1 to case
6 of Table \ref{recapitulatif}. Thus, when external disturbances
(such as herbivory) on grass biomass are low and $\sigma_G$ has
relatively large values, the IFAC model predicts either a stable
savanna state, a stable forest state, a stable grassland state or a
multistability involving savanna and/or forest and/or grassland (see
also case 1, case 2, case 3, case 5 and case 6 of Table
\ref{recapitulatif}). Furthermore, when external disturbances on
grass biomass become more important, the grassland solution becomes
unstable and the IFAC model predicts either a stable savanna state,
a stable forest state or a bistability involving savanna and foret
states (see also case 1 and case 4 of Table \ref{recapitulatif}).
Consequently, one can observe that the non sensitive tree vs. grass
interaction parameter $\sigma_{NS}$
 and the additional death rate of grass biomass due to external
disturbances $\mu_G$ are likely to be influential on the IFAC
outcomes in Region 1. Nevertheless, with a mean annual rainfall of
300-400 mm there is a wide array of references evidencing the
probable bistability of desert (bare soil) and thickets (in the
African Sahel) (see Couteron \& Kokou (1997) \cite{Couteron1997},
Lefever et al. (2009) \cite{Lefever2009}, Barbier et al. (2008)
\cite{Barbier2008}) or desert and grass (Namibia) (see Tschinkel
(2012) \cite{Tschinkel2012}, Fernandez-Oto et al. (2014)
\cite{Fernandez2014}).

\subsection{Results for Region 2}
In Region 2 which corresponds to a mesic area, the main mechanisms
that regulate tree-grass interactions also include mechanisms
\textbf{(M1)} and \textbf{(M2)} stated previously for semiarid
areas. In addition to \textbf{(M1)} and \textbf{(M2)}, in mesic
areas, fires are more frequent than in semiarid areas since water
availability favor grass biomass production which constitutes the
fuel for fires (we denote this new mechanism \textbf{(M4)}).
Grass-fire feedback (mechanism \textbf{M4}) together with mechanisms
\textbf{(M1)} and \textbf{(M2)}, maintain both forest and savanna
occurrences in mesic areas (Baudena et al. (2014)
\cite{Baudena2014}). Indeed, grasses benefit from the openness of
the landscape after fires, since they recover faster than trees
seedlings, thus determining a positive feedback mechanism that
enhances savanna presence. The IFAC model also predict a shift from
a forest state to a savanna state as $\sigma_{NS}$ decreases and/or
when  $\sigma_{G}$  increases, which also agree with Baudena et al.
(2014) \cite{Baudena2014} results. Therefore, savanna and forest
vegetation types clearly appear as alternatively stable states as
found by Staver et al. (2011) \cite{Staver2011}, Staver and Levin
(2012) \cite{Staver2012}, in our case depending on $\sigma_{G}$ and
$\sigma_{NS}$ variation in Region 2. \\Let us consider the following
table of parameters values
  \begin{table}[H]
  \centering
  \caption{Array of parameters' values for Region 2}\label{R2}
  \vspace{0.25cm}
  \small
  \begin{tabular}{|c|c|c|c|c|c|c|c|}
    \hline
 $\gamma_S$ & $\gamma_{NS}$ & $\gamma_G$ & $\mu_{NS}$ &  $\omega_S$ & $\eta_S$ & $\eta_G$ & $\mu_{S}$\\
     \hline\hline
0.4 & 2 & 2.8 & 0.08 & 0.1 & 0.5 & 0.6 & 0.1\\
    \hline\hline
       \multicolumn{8}{|c|}{$K_T$=85,~~ $K_G=7$,~~$\mu_{G}=0.3$ }\\
    \hline
  \end{tabular}
\end{table}

Using parameters values in Table \ref{R2}, one has:
\begin{center}
\begin{tabular}{|c|c|}
  \hline
$\mathcal{R}_T^0$  & $\mathcal{R}_G^0$\\
    \hline
 14.5  &  14\\
  \hline
\end{tabular}
\end{center}

and we also derive figure \ref{R2fig3} and figure \ref{R2fig1}.

\begin{figure}[H]
    \centering
    \resizebox{1\textwidth}{!}{\includegraphics{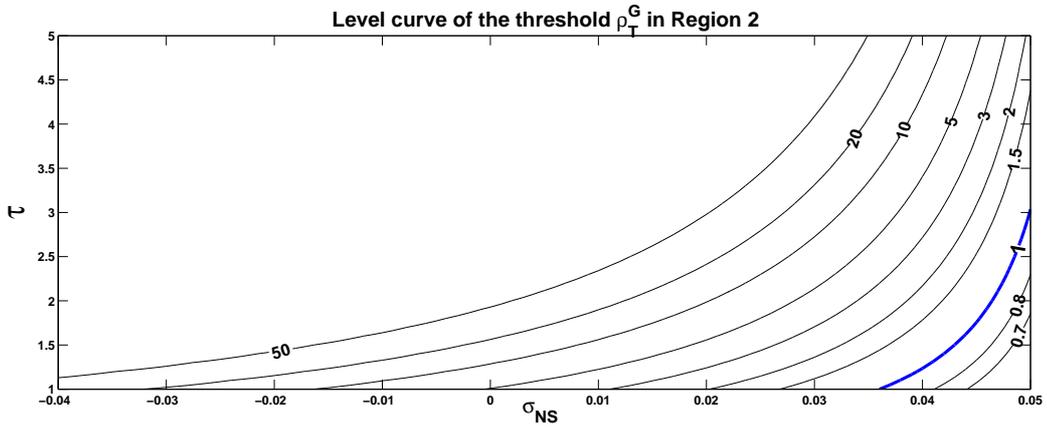}}
 \caption{Level curve of the threshold $\rho_T^G$ illustrating that system (\ref{fa})-(\ref{ifa}) is liable to move
 from a savanna/grassland state to a forest state or to a multistability involving the forest solution together with $\tau$ and $\sigma_{NS}$ variations.
  Recall that the forest solution is stable (resp. unstable) whenever $\rho_T^G$ is lower (resp. greater) than unity.} \label{R2fig3}
 \end{figure}

\begin{figure}[H]
    \centering
    \resizebox{1\textwidth}{!}{\includegraphics{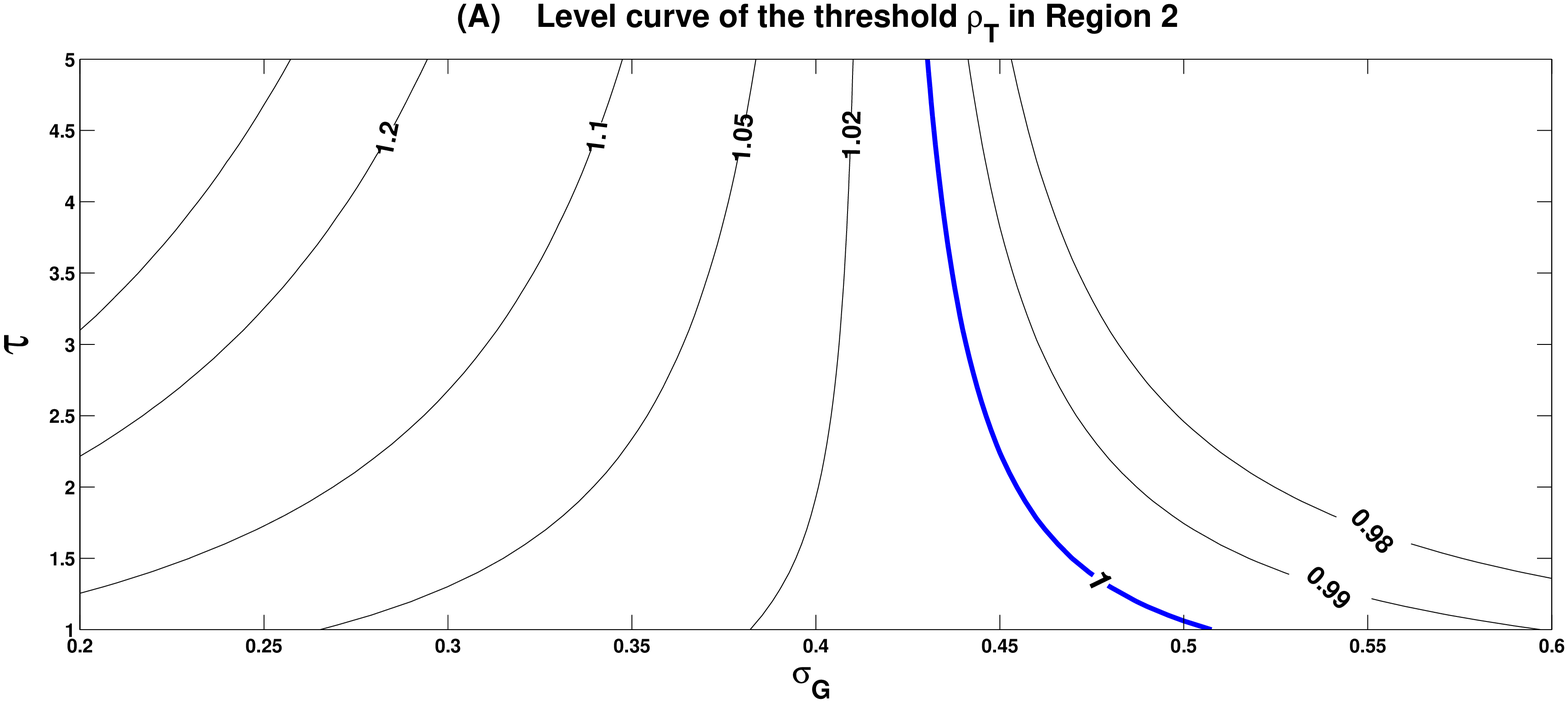}}
    \resizebox{1\textwidth}{!}{\includegraphics{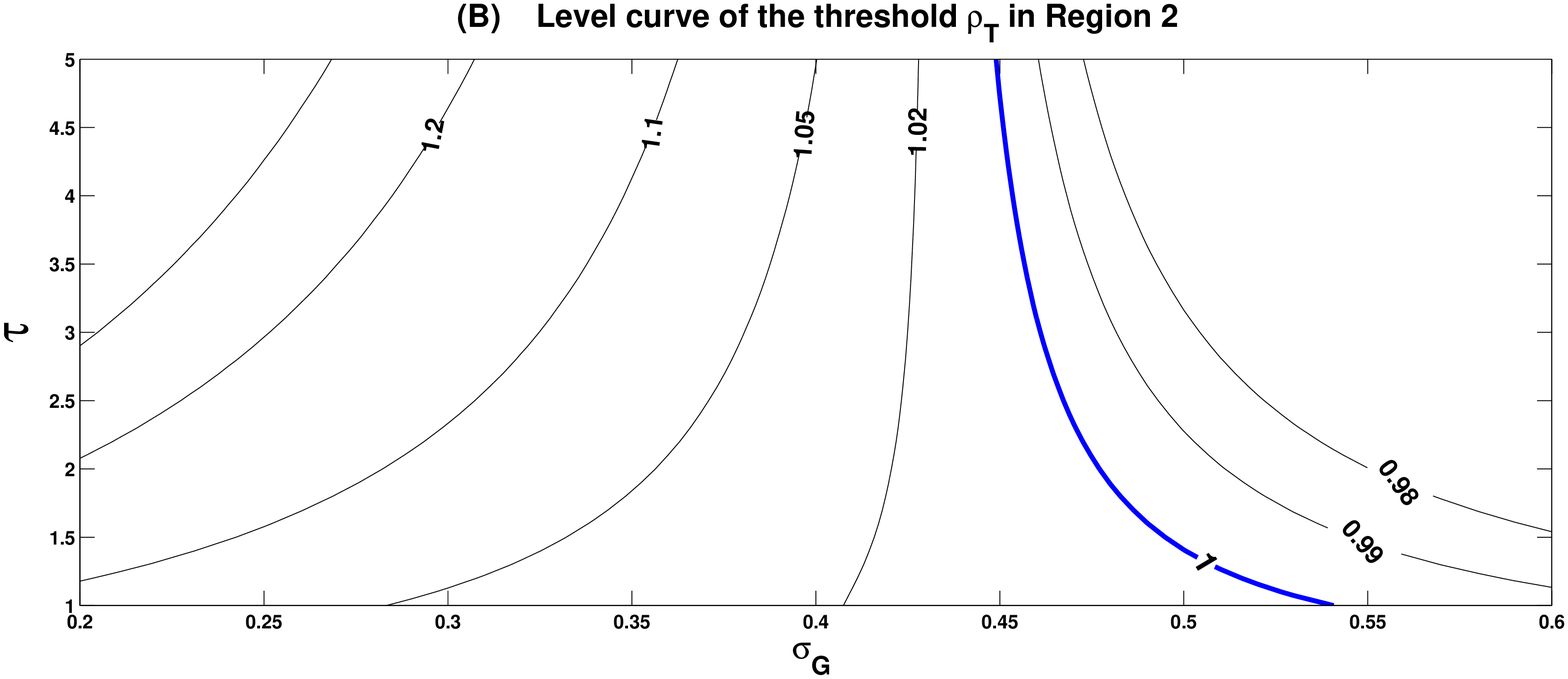}}
 \caption{Level curve of the threshold $\rho_T$ illustrating that system (\ref{fa})-(\ref{ifa}) is liable to move
 from a savanna/forest state to a grassland state or to a multistability involving the grassland solution in relation to $\tau$ and $\sigma_{G}$ variations.
  Recall that the grassland solution is stable (resp. unstable) whenever $\rho_T$ is lower (resp. greater) than unity. In (A) $\mu_G=0.2$, in (B) $\mu_G=0.3$.} \label{R2fig1}
 \end{figure}

Since for parameters values in Table \ref{R2} one has
$\mathcal{R}_G^0>1$ then,
$$\rho_G^0>1\Longleftrightarrow \tau> \left\{\begin{array}{cl}
                                                 0.3524& \mbox{for $\mu_G=0.2$} \\
                                                 0.3665& \mbox{for
                                                 $\mu_G=0.3$}.
                                               \end{array}\right.
 $$
 Moreover, for our estimation of $\sigma_{NS}$ and $\sigma_{G}$ one also has $\mathcal{R}_T^G>1$ and $\mathcal{R}_G^T>1$.
 Therefore, figure \ref{R2fig1} and figure \ref{R2fig3} illustrate, either case 1,
case 2, case 4 or case 5 of Table \ref{recapitulatif}.\par

In summary, the parameters that are likely to be influential on the
IFAC outcomes in Region 2 are the external disturbances on grass
biomass parameter $\mu_{G}$, the grass vs. sensitive tree parameter
$\sigma_G$, the non sensitive tree vs. grass interaction parameter
$\sigma_{NS}$. In addition to that previous parameters, one can also
mention the fire return time $\tau$. 
In other words, the previous analysis reveals that in Region 2, in
addition of stability of forest, stability of savanna and
bistability of forest and savanna as in Region 1, one can observe
stability of grassland and also multistabilty situations involving
grassland solution with relatively low values of $\sigma_G$ in
comparison with Region 1.

\subsection{Results for Region 3}
Region 3 corresponds to humid tropical areas where rainfall
availability favors biomass production of both woody and grasses
components. The grass-fire feedback possibly leads to stability of
either savanna or forest in Region 3 depending on fire return time.
Indeed, grass, particularly abundant in these wet areas, becomes an
extremely good fuel in the dry season, which promotes fire
occurrence and increases fire intensity and impact (Baudena et al.
(2014) \cite{Baudena2014}, Higgins et al. (2008)
\cite{Higgins2008}). When the fire return time is large, the trees
have the time to grow above the flame zone and to reach canopy
closure and then outcompete grasses. Therefore, relatively large
return time favor forest state in Region 3 (Staver and Levin (2012)
\cite{Staver2012}). Moreover, if the fire return time is small then
trees don't have the time to reach canopy closure and therefore let
grasses  which regrow quickly in the open space after fires
 form either a stable savanna state or a stable
grassland state. We illustrate hereafter that these features are
also predicted by the IFAC model.

Consider
  \begin{table}[H]
  \centering
  \caption{Array of parameters' values for Region 3}\label{R3}
  \small
  \begin{tabular}{|c|c|c|c|c|c|c|c|c|}
    \hline
 $\gamma_S$ & $\gamma_{NS}$ & $\gamma_G$  & $\mu_{NS}$ &  $\omega_S$ & $\eta_S$ & $\eta_G$ & $\mu_{S}$ & $\mu_{G}$\\
     \hline\hline
2 & 3 & 4.2  & 0.06 & 0.1 & 0.5 & 0.6 & 0.1 & 0.2\\
    \hline\hline
       \multicolumn{9}{|c|}{ $K_T$=115,~~ $K_G=15$}\\
    \hline
  \end{tabular}
\end{table}
In this section, we will refer to a particular area, namely the
Lamto region in Ivory Coast (see Menaut et al. (1979)
\cite{Menaut1979}, Mordelet \& Menaut (1995) \cite{Mordelet1995}).
Thanks to Abadie et al. (2006) Page 156 \cite{Abbadie2006} (see also
Mordelet \& Menaut (1995) \cite{Mordelet1995}) data report of grass
biomass in the canopy and open situations in Lamto and by
reinterpreting their results we have derived the range
$0.0609\leq\sigma_{NS}\leq0.0913$.\par

Using parameters values in Table \ref{R3} one has:
\begin{center}
\begin{tabular}{|c|c|}
  \hline
 $\mathcal{R}_T^0$  & $\mathcal{R}_G^0$ \\
    \hline
  35  & 21 \\
  \hline
\end{tabular}
\end{center}

We also derived the following figure \ref{R3fig1}.
\begin{figure}[H]
    \centering
    \resizebox{1\textwidth}{!}{\includegraphics{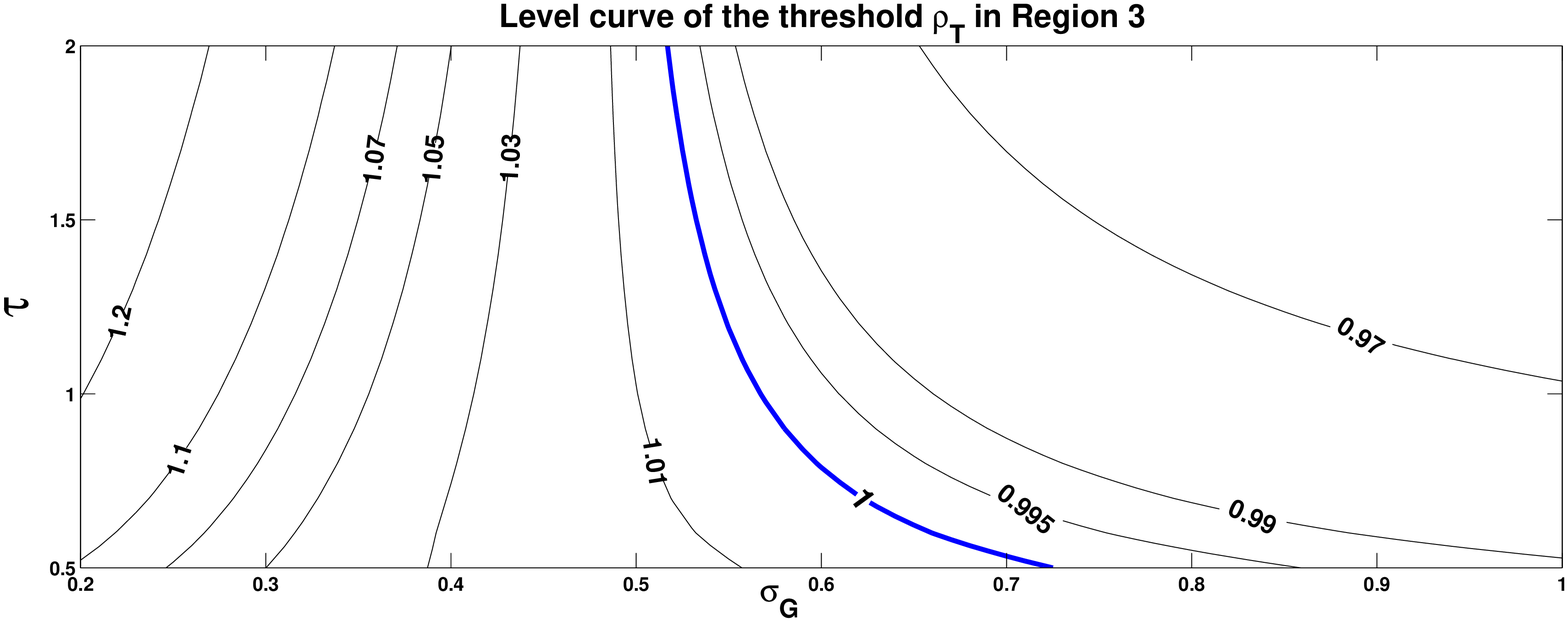}}
 \caption{Level curve of the threshold $\rho_T$ illustrating that system (\ref{fa})-(\ref{ifa}) is liable to move
 from a savanna/forest state to a grassland state or to a multistability involving the grassland solution in relation to $\tau$ and $\sigma_{G}$ variations.
  Recall that the grassland solution is stable (resp. unstable) whenever $\rho_T$ is lower (resp. greater) than unity.} \label{R3fig1}
\end{figure}

 Since for parameters values in Table
\ref{R3} and according to our estimation of $\sigma_{NS}$ one has
\begin{itemize}
    \item[$\bullet$] $\mathcal{R}_T^G<1$, $\rho_T^G<1$,
    \item[$\bullet$] $\mathcal{R}_G^0>1$ then,
$$\rho_G^0>1\Longleftrightarrow \tau>0.2291.$$
Values of $\tau$ are not expected  to be under 0.5 (i.e. 2 fires per
year), a minimum which corresponds to sub-equatorial climates with
two dry seasons. We therefore consider that this condition is always
fulfilled.
\end{itemize}
 Therefore, one can
deduce that figure \ref{R3fig1} illustrates either case 8, 9 or case
10 of Table \ref{recapitulatif}.\par

The previous analysis highlighted the importance of the grass vs.
sensitive tree competition parameter, $\sigma_G$, the  non sensitive
tree vs. grass  competition parameter, $\sigma_{NS}$, and the fire
return time, $\tau$, in controlling the outcomes of the IFAC model.
Comparing to results of Region 2 (in terms of having the forest, the
grassland or the savanna as reachable solution), one note that the
IFAC model fairly has the same outcomes in Region 3 as in Region 2.
Figure \ref{R3fig1} further show that for low values of $\sigma_G$,
say less than 0.6, $\tau$ seems to have very limited influence. It
is stronger for large values for which increases in $\tau$ make
$\rho_T$ decreases under 1 and therefore destabilize the grassland
solution.

\par

In summary, in Region 3, we observed that the fire return time along
with the grass vs. sensitive tree competition parameter $\sigma_G$
and the non sensitive tree vs. grass competition parameter
$\sigma_{NS}$ strongly influence the outcome of the IFAC model.
Indeed, depending on these parameters variations and values, the
IFAC can converge either to a grassland state, to a savanna state,
to a forest state or to a multistability involving forest state and
either grassland or savanna while environmental conditions in this
Region would systematically allow forests in the absence of fire
(Staver and Levin (2012) \cite{Staver2012}). Therefore to favor a
forest state in Region 3, one could implement policies in order to
have relatively large fire return time (says a fire frequency, f,
lower than one fire per year: $f<1$) and vice versa if tracts of
savanna are to be kept against forest encroachment as habitats of
large grazing mammals.

\section{Conclusion }\label{conclusion}
In this work, we presented and analyzed a new kind of mathematical
model for tree-grass interactions in savanna ecosystems either fire
prone or not. It is an extension of a continuous-time model, called
the COFAC model, studied in Yatat et al. (2014) \cite{Yatat2014}.
The model presented here, that we call the IFAC model, is based on a
system featuring impulsive differential equations and thereby aims
to acknowledge the discrete nature of fire events. The analytical
study of the IFAC reveals a desert equilibrium, a forest equilibrium
and two periodic solutions: the grassland periodic solution and the
savanna periodic solution. The analytical study also reveals seven
ecological thresholds ($\mathcal{R}_T^0$, $\mathcal{R}_G^0$,
$\rho_G^0$, $\mathcal{R}_G^T$, $\mathcal{R}_T^G$, $\rho_T^G$,
$\rho_T$). These thresholds define in parameter space regions of
monostability, bistability also found with the models of Accatino et
al. (2010) \cite{Accatino2010}, De Michele et al. (2011)
\cite{Demichele2011}, Yatat et al. (2014) \cite{Yatat2014}. They
also define regions of tristability as in Yatat et al. (2014)
\cite{Yatat2014} with respect to the equilibria (desert and forest)
and periodic solutions (grassland and savanna). The specificity of
IFAC is to also present periodic behaviors which depict fluctuations
in woody and/or grassy biomass and cannot be yielded by fully
continuous time models. Therefore, in case of transition from a
vegetation type (forest, savanna or grassland) to another, the
change is done progressively which is more ecologically meaningful
than the abrupt changes observed for fires-continuous models such as
in Accatino et al. (2010) \cite{Accatino2010}, De Michele et al.
(2011) \cite{Demichele2011} and Yatat et al. (2014)
\cite{Yatat2014}. Such abrupt changes have been criticized as
unrealistic by Accatino et al. (2013) \cite{Accatino2013} (see also
Beckage et al. (2011) \cite{Beckage2011}) who advocated stochastic
models in lieu of continuous time formulation of fire impact on
vegetation. The present approach however demonstrates that a more
realistic modeling of fire can be introduced within the framework of
continuous time models while keeping the potential for analytical
exploration of the main outcomes of the model. Something which is
not possible with the aforementioned fully stochastic models.

\par

As in Yatat et al. (2014) \cite{Yatat2014}, we found that the
competition parameters $\sigma_G$ which expresses the asymmetric
competition exerted by grasses on sensitive trees (shading and soil
resource preemption) and $\sigma_{NS}$ that expresses the asymmetric
competition of non sensitive trees on grasses (shading and soil
resource preemption) are bifurcation parameters of the IFAC model
along with fire frequency (which strongly influence the convergence
outcomes of the IFAC model). The analytical  study of the IFAC model
also reveals three particular values $\tau^\star$,
$\sigma_{NS}^\star$ and $\sigma_G^\star$ (see relation (\ref{TGF}),
Page \pageref{TGF}) that delimit regions of stability/instability or
forest and grassland solutions in relation to $\tau$, $\sigma_G$ and
$\sigma_{NS}$ respectively
 (see relations (\ref{SG}), Page \pageref{SG} and (\ref{SF}), Page
\pageref{SF}). Moreover, considering three ecological biomass
production zones indexed by fires frequency and by carrying
capacities of both trees and grasses biomass, allows us to point out
several scenarios for IFAC convergence that depend on $\sigma_G$,
$\sigma_{NS}$ and $\tau$ values and variations. These outcomes of
the IFAC model are qualitatively in agreement with results of
Baudena et al. (2014) \cite{Baudena2014}, Staver et al. (2011)
\cite{Staver2011}, February et al. (2013) \cite{February2013},
Accatino et al. (2010) \cite{Accatino2010}, Couteron \& Kokou (1997)
\cite{Couteron1997}, Staver \& Levin (2012) \cite{Staver2012}.
Distinguishing three ecological zones allowed us to verify in which
contexts the possible bifurcation parameters are actually
influential or not. This analysis highlighted the pervasiveness of
$\sigma_{NS}$ (i.e. the depressive or facilitation effect of
grown-up trees on grasses) in all the three zones. It also
emphasized the influence of $\sigma_{G}$ (i.e. depressive effect of
grasses on small trees) in the two zones (2 and 3) with sufficient
rainfall to allow medium to high grass production. In these two
zones, and especially in zone 3, the fire return time ($\tau$)
appeared also influential. As already mentioned by Yatat et al.
(2014) \cite{Yatat2014}, and verified here for the IFAC model, the
competition parameters $\sigma_G$ and $\sigma_{NS}$ which embody
direct tree-grass interactions deserve an increased interest and
should be the focus of adhoc observations and experiments as to
better assess their ranges of variation in the different ecological
regions.
\par
Although the IFAC model presented in this work and the COFAC model
presented in Yatat et al. (2014) \cite{Yatat2014} qualitatively
display strong similarities, the IFAC model is richer. Indeed
modelling fire events as pulse phenomena leads to a relaxation of
stability conditions of both forest and grassland solutions and it
increases parameters ranges for which bistability situations
involving forest, grassland and savanna can occur. This particular
property of the IFAC model may explain, along with bifurcation
parameters and its periodic outcomes, many changes in tree-grass
interactions in fire-prone ecosystems. Thanks to this particular
property, the IFAC model, which moreover displays periodic outcomes
and bifurcations according to well-identified parameters, is able to
account for many dynamical scenarios observed in savanna-like
ecosystems from the fringes of the desert to the boundary of the wet
forest.

\vspace{0.5cm}

\noindent \textbf{Acknowledgements}\\
The first author is grateful to the French governement and the French Embassy in Yaound\'e (Cameroon) for their support (SCAC fund) during the preparation of this manuscript.

\appendix

\section*{Appendix A: Proof of Lemma \ref{stabiliteEG}}
 Defining
 \begin{equation}
    \begin{array}{lcl}
      T_S(t) & = & x(t), \\
      T_{NS}(t) & = &  y(t), \\
      G(t) & = & G^*(t) + z(t),
    \end{array}
 \end{equation}
where $x(t), y(t)$ and $z(t)$ are small perturbations. Every
solution of the linearized equations can be written as
\begin{equation}
    \left(
      \begin{array}{c}
        x(t) \\
        y(t) \\
        z(t) \\
      \end{array}
    \right)=\Phi(t)\left(
                     \begin{array}{c}
                       x(0) \\
                       y(0) \\
                       z(0) \\
                     \end{array}
                   \right).
\end{equation}
Here $\Phi$ is a fundamental matrix and satisfies,
\begin{equation}
    \begin{array}{ccl}
       \displaystyle\frac{d\Phi(t)}{dt}&=&DF(0; 0; G^*(t))\Phi(t)\\
       & = & \left(
                                                   \begin{array}{ccc}
                                                     \gamma_S-(\mu_S+\omega_S+\sigma_GG^*(t)) & \gamma_{NS} & 0 \\
                                                     \omega_S & -\mu_{NS} & 0 \\
                                                     0 & -\sigma_{NS}G^*(t) & \gamma_G-2\displaystyle\frac{\gamma_G}{K_G}G^*(t)-\mu_G \\
                                                   \end{array}
                                                 \right)
    \Phi(t)
    \end{array}
\end{equation}

 and
$\Phi(0)=Id_{\mathbb{R}^3}$. Moreover the resetting impulsive
condition of system (\ref{fa})-(\ref{ifa}) becomes,
\begin{equation}
                        \left(
                          \begin{array}{c}
                            x(n\tau^+) \\
                            y(n\tau^+) \\
                            z(n\tau^+) \\
                          \end{array}
                        \right)
                        =\left(
                           \begin{array}{ccc}
                             1-\eta_Sw(G^*(\tau)) & 0 & 0 \\
                             0 & 1 & 0 \\
                             0 & 0 & 1-\eta_G \\
                           \end{array}
                         \right)\left(
                          \begin{array}{c}
                            x(n\tau) \\
                            y(n\tau) \\
                            z(n\tau) \\
                          \end{array}
                        \right).
\end{equation}
A monodromy matrix $\mathbf{M}$ of system $(\ref{fa})-(\ref{ifa})$,
is:
\begin{equation}\label{mg}
    \mathbf{M}=\left(
                           \begin{array}{ccc}
                             1-\eta_Sw(G^*(\tau)) & 0 & 0 \\
                             0 & 1 & 0 \\
                             0 & 0 & 1-\eta_G \\
                           \end{array}
                         \right)\Phi(\tau),
\end{equation}
with
\begin{equation}\label{phig}
    \begin{array}{ccc}
      \Phi(t) & = & \exp\left(\int_0^tDF(0; 0;
      G^*(s))ds\right).
    \end{array}
\end{equation}
Moreover using Lemma \ref{lemme_integrale}, a direct computations
leads $$
  \int_0^\tau DF(0; 0;
      G^*(s))ds  =\left(
                    \begin{array}{ccc}
                      DF^{(1)} & DF^{(2)} & 0 \\
                      DF^{(3)} & DF^{(4)} & 0 \\
                      0 & DF^{(5)} & DF^{(6)} \\
                    \end{array}
                  \right),
      $$
where
\begin{equation}
\begin{array}{ccl}
  DF^{(1)} & = &
  (\gamma_S-(\mu_S+\omega_S))\tau-\sigma_G\displaystyle\int_{0}^{\tau}G^*(t)dt,\\
  DF^{(2)} & = & \gamma_{NS}\tau, \\
  DF^{(3)} & = & \omega_S\tau, \\
  DF^{(4)} & = & -\mu_{NS}\tau, \\
  DF^{(5)} & = & -\sigma_{NS}\displaystyle\int_{0}^{\tau}G^*(t)dt,\\
  DF^{(6)} & = &-\mu_G(\mathcal{R}_G^0-1)\tau-2\ln(1-\eta_G).
\end{array}
\end{equation}
Consider the sub-matrix $\mathbf{B}$ defined as follow:
\begin{equation}
\mathbf{B}=\left(
    \begin{array}{cc}
      DF^{(1)} & DF^{(2)} \\
      DF^{(3)} & DF^{(4)} \\
    \end{array}
  \right).
\end{equation}
Recall that eigenvalues of the matrix $\mathbf{B}$ are root of the
quadratic equation
$$\lambda^2- trace (\mathbf{B}) \lambda+ \det(\mathbf{B})=0$$

and to characterize real part of eigenvalues of matrix $\mathbf{B}$,
following Routh-Hurwitz criterium (see Section 1.3.5 Page 72 of
Augier et al. (2010) \cite{Augier2010}), one need only to study the
sign of $trace (\mathbf{B})$ and $\det(\mathbf{B})$.
 Let
\begin{equation}
\begin{array}{ccl}
 \mathcal{A}= tr(\mathbf{B}) & = & \gamma_S\tau\left(1-\displaystyle\frac{1}{\mathcal{R}}\right),
\end{array}
\end{equation}
where
$$\mathcal{R}=\displaystyle\frac{\gamma_S}{\mu_S+\omega_S+\mu_{NS}+\sigma_GG_{int}}>0.$$
Thus, if $\mathcal{R}<1$ then $\mathcal{A}<0$.\\
Moreover, let
\begin{equation}
\begin{array}{ccl}
  \mathcal{B}=\det(\mathbf{B}) & = & \tau\mu_{NS}\left((\mu_S+\omega_S)\tau+\sigma_GG_{int}\right)\left(1-\mathcal{R}_G^T\right),
\end{array}
\end{equation}
where
$$\mathcal{R}_G^T=\displaystyle\frac{\gamma_S\mu_{NS}+\omega_S\gamma_{NS}}{\mu_{NS}(\mu_S+\omega_S)+\sigma_G\mu_{NS}G_{int}}>0.$$
Thus, if $\mathcal{R}_G^T<1$ then $\mathcal{B}>0$.\\
Moreover, one also has $\mathcal{R}<\mathcal{R}_G^T.$\\
 Therefore, if
$\mathcal{R}_G^T<1$ then $s(\mathbf{B})<0$, where $s$ denotes the
stability modulus (i.e. the maximum of the real part of eigenvalues).\\
From expressions $(\ref{mg})$ and $(\ref{phig})$ we deduced that
eigenvalues $\xi_1$, $\xi_2$ and $\xi_3$ of the monodromy matrix
$\mathbf{M}$ are
\begin{equation}
\begin{array}{ccl}
  \xi_1 & = & (1-\eta_Sw(G^*(\tau)))e^{\lambda_1}, \\
  \xi_2 & = & e^{\lambda_2}, \\
  \xi_3 & = &\displaystyle\frac{e^{-\mu_G(\mathcal{R}_G^0-1)\tau}}{1-\eta_G},
\end{array}
\end{equation}
where $\lambda_1$, $\lambda_2~\in~sp(\mathbf{B})$.\\
 Since
$0<1-\eta_Sw(G^*(\tau))\leq1$, if $\mathcal{R}_G^T<1$, then $0\leq
\xi_1<1$ and $0< \xi_2<1.$\\ Moreover, since $\rho^0_G>1$ then
$\xi_3<1$. Indeed,
 $$\xi_3<1 \Leftrightarrow
e^{-\mu_G(\mathcal{R}_G^0-1)\tau}<1-\eta_G \Leftrightarrow
1<(1-\eta_G)e^{\mu_G(\mathcal{R}_G^0-1)\tau}\Leftrightarrow
\rho^0_G>1.$$ Finally we deduce that the grassland periodic solution
$E_G=(0; 0; G^*(t))$ is locally asymptotically stable if
$\mathcal{R}_G^T<1$ or $(\mathcal{R}_G^T>1$ and $\rho_T<1)$, is
locally stable if $(\mathcal{R}_G^T>1$ and $\rho_T=1)$ and is
unstable if $(\mathcal{R}_G^T>1$ and $\rho_T>1)$. This ends the
proof.

\section*{Appendix B: Proof of Theorem \ref{prop}}
$\star$ Case 1: $\mu_G>0$.\\
Solution $G$ of system
(\ref{fa})-(\ref{ifa}) satisfy
\begin{equation}
    \begin{array}{ccc}
      G'(t) & \leq & \gamma_G\left(1-\displaystyle\frac{1}{\mathcal{R}_G^0}\right)G(t) \\
      G(t_k^+) & = & (1-\eta_G)G(t_k).
    \end{array}
\end{equation}
From Lemma 1.3 page  15 in \cite{Bainov1995} we deduce that
$$G(t)\leq
    G(0)\left(\prod\limits_{0\leq
    t_k<t}(1-\eta_G)\right)\exp\left(\gamma_G\left(1-\displaystyle\frac{1}{\mathcal{R}_G^0}\right)t\right).$$Thus,
    for $\mathcal{R}_G^0<1$ we have $\lim\limits_{t\rightarrow
    +\infty}G(t)=0$ and solutions $T_S$ and $T_{NS}$ of system
    (\ref{fa})-(\ref{ifa}) satisfy
    \begin{equation}\label{tstns}
    \left\{
        \begin{array}{ccl}
          T_S' & = & (\gamma_ST_S+\gamma_{NS}T_{NS})\left(1-\displaystyle\frac{T_S+T_{NS}}{K_T}\right)-T_S(\mu_S+\omega_S), \\
          T'_{NS} & = & \omega_ST_S-\mu_{NS}T_{NS}.
        \end{array}
        \right.
    \end{equation}
System (\ref{tstns}) does not admit periodic solution (see Appendix
B in Yatat et al. (2014) \cite{Yatat2014}), thus using the jacobian
matrix of system (\ref{tstns}) we deduce that
\begin{itemize}
    \item if $\mathcal{R}_T^0<1$ then, $(T_S, T_{NS})\rightarrow
    (0,0)$,
    \item if
    $\mathcal{R}_T^0>1$ then, $(T_S, T_{NS})\rightarrow (\bar{T}_S,\bar{T}_{NS})$, where $(\bar{T}_S,\bar{T}_{NS})$ are given
    in (\ref{ts}).
\end{itemize}
At the end, we deduce that if $\mathcal{R}_T^0<1$ and
$\mathcal{R}_G^0<1$ then, the desert equilibrium $E_0$ is GAS i.e.,
point 1 of Theorem \ref{prop} holds. The forest equilibrium $E_T$ is
GAS whenever $\mathcal{R}_T^0>1$ and $\mathcal{R}_G^0<1$ i.e., point
2 of Theorem \ref{prop} holds.\par Now suppose that
$\mathcal{R}_T^0<1$ and $\mathcal{R}_G^0>1$. Solutions $T_S$ and
$T_{NS}$ of system (\ref{fa})-(\ref{ifa}) satisfy
\begin{equation}
 \left\{%
\begin{array}{lclcr}
  \displaystyle\frac{dT_S}{dt} &\leq& (\gamma_ST_S+\gamma_{NS}T_{NS})\left(1-\displaystyle\frac{T_S+T_{NS}}{K_T}\right)-T_{S}(\mu_S+\omega_S),& & \\
  & & & & \\
  \displaystyle\frac{dT_{NS}}{dt} &\leq& \omega_ST_S-\mu_{NS}T_{NS},& & t\neq t_k\\
  \end{array}
\right.
\end{equation}
\begin{equation}
 \left\{%
\begin{array}{lclcr}
  T_S(t_k^+) &\leq& T_S(t_k),& & \\
  T_{NS}(t_k^+) &\leq& T_{NS}(t_k),&t=t_k & t_{k+1}=t_k+\tau.
\end{array}
\right.
\end{equation}
Let consider the upper system
\begin{equation}
 \left\{%
\begin{array}{lcl}
  \displaystyle\frac{du}{dt} &=& (\gamma_Su+\gamma_{NS}v)\left(1-\displaystyle\frac{u+v}{K_T}\right)-u(\mu_S+\omega_S),\\
  & & \\
  \displaystyle\frac{dv}{dt} &=& \omega_Su-\mu_{NS}v,\\
  \end{array}
\right.
\end{equation}
Since $\mathcal{R}_T^0<1$, $(u(t),v(t))\rightarrow (0,0)$. Thus
$(T_S(t), T_{NS}(t))\rightarrow (0,0)$. Furthermore, solution $G$ of
system (\ref{fa})-(\ref{ifa}) admits as limiting system
\begin{equation}\label{herbe1}
 \left\{%
\begin{array}{lclc}
  \displaystyle\frac{dG}{dt} &=& \gamma_G\left(1-\displaystyle\frac{G}{K_G}\right)G-\mu_GG& t\neq t_k,\\
  G(t_k^+)&=&(1-\eta_G)G(t_k) & t=t_k.
\end{array}
\right.
\end{equation}
System (\ref{herbe1}) admits at most two solutions: the trivial
solution, 0, which always exists and the periodic solution $G^*(t)$
which is ecologically meaningful if $\rho_G^0>1$
 where
$G^*(t)$ is given by (\ref{getoile}). Now we turn to check stability
results of solutions of system (\ref{herbe1}) through small
perturbations approach and Floquet's theory.
\begin{itemize}
    \item Setting $G(t)=x(t)$ where $x$ is a small perturbation and
    verify $x(t)=\phi(t)x_0$, where $\phi$ verify $$\phi'(t)=\mu_G\left(\mathcal{R}_G^0-1\right)\phi(t)$$
    and $\phi(0)=1.$ The resulting impulsive condition becomes
    $$x(nT^+)=(1-\eta_G)x(nT).$$ Following the Floquet's theory, the
    zero equilibrium is locally asymptotically stable if $$\lambda_0=(1-\eta_G)e^{\mu_G\left(\mathcal{R}_G^0-1\right)\tau}<1.$$
    Since\\ $\mathcal{R}_G^0>1$,
    $\lambda_0<1$
    if and only if $(1-\eta_G)e^{\mu_G\left(\mathcal{R}_G^0-1\right)\tau}<1$,
    i.e.
     $\rho_G^0<1.$\\
     Moreover,
    for $\rho_G^0<1$ the
    positive solution $G^*(t)$ is undefined then the desert
    equilibrium is globally asymptotically stable. Finally, we deduce that the desert solution $(0,0,0)$ is globally asymptotically stable whenever
     $\mathcal{R}_T^0<1$, $\mathcal{R}_G^0>1$ and $\rho_G^0<1$. Point 3 of
     Theorem
    \ref{prop} holds.
    \item Now, setting $G(t)=G^*(t)+x(t)$ where $x$ is a small perturbation and
    verify $x(t)=\phi(t)x_0$, where $\phi$ verify $$\phi'(t)=\left[\mu_G\left(\mathcal{R}_G^0-1\right)-\displaystyle\frac{2\gamma_G}{K_G}G^*(t)\right]\phi(t)$$
    and $\phi(0)=1.$ The resulting impulsive condition becomes
    $$x(nT^+)=(1-\eta_G)x(nT).$$ According to the Floquet's theory,
    solution $G^*(t)$ is locally asymptotically stable if $$\lambda_{G^*}=(1-\eta_G)\exp\left\{\mu_G\left(\mathcal{R}_G^0-1\right)\tau-\displaystyle\frac{2\gamma_G}{K_G}\int_{n\tau}^{(n+1)\tau}G^*(t)dt\right\}<1.$$
Following Lemma \ref{lemme_integrale},
\begin{equation}
    \begin{array}{ccc}
      \displaystyle\int_{n\tau}^{(n+1)\tau}G^*(s)ds & = &
      \displaystyle\frac{K_G}{\gamma_G}\left\{\ln(1-\eta_G)+\mu_G\left(\mathcal{R}_G^0-1\right)\tau\right\}.
    \end{array}
\end{equation}
Thus
\begin{equation}
    \begin{array}{ccl}
    \lambda_{G^*}  & = & (1-\eta_G)\exp\left\{-\mu_G\left(\mathcal{R}_G^0-1\right)\tau-2\ln(1-\eta_G)\right\} \\
       & = &
       \exp\left\{-\mu_G(\mathcal{R}_G^0-1)\tau-\ln(1-\eta_G)\right\}.
    \end{array}
\end{equation}
Since $\rho_G^0>1$, we have:
\begin{equation}
\begin{array}{ccl}
  \rho_G^0>1 & \Leftrightarrow & (1-\eta_G)\exp\{\mu_G(\mathcal{R}_G^0-1)\tau\}>1 \\
   & \Leftrightarrow & \ln(1-\eta_G)>-\mu_G(\mathcal{R}_G^0-1)\tau \\
   & \Leftrightarrow & -\mu_G(\mathcal{R}_G^0-1)\tau-\ln(1-\eta_G)<0
\end{array}
\end{equation}
and we deduce
$$\lambda_{G^*}<1.$$
Thus solution $G^*(t)$ of (\ref{herbe1}) is globally asymptotically
stable because the zero solution, in this case, is unstable.
Finally, we deduce that the grassland periodic solution $(0,0,
G^*(t))$ is globally asymptotically stable whenever
$\mathcal{R}_T^0<1$, $\mathcal{R}_G^0>1$ and $\rho_G^0>1$. Point 4
of Theorem \ref{prop} holds.
\end{itemize}

$\star$ Case 2: $\mu_G=0$.\\
The proof of points (i) and (ii) of Theorem \ref{prop} is fairly the
same as the proof of points 3 and 4. Indeed we first set, only in
System (\ref{herbe1}), $\mu_G=0$ and next, we substitute
$\mu_G(\mathcal{R}_G^0-1)$ by $\gamma_G$ in the rest of
the proof.\\
Solution $G$ of system (\ref{fa})-(\ref{ifa}) satisfy
\begin{equation}\label{herbe2}
 \left\{%
\begin{array}{lclc}
  \displaystyle\frac{dG}{dt} &\leq& \gamma_G\left(1-\displaystyle\frac{G}{K_G}\right)G& t\neq t_k,\\
  G(t_k^+)&=&(1-\eta_G)G(t_k) & t=t_k.
\end{array}
\right.
\end{equation}
Since $\rho_G^0= (1-\eta_G)\exp\{\gamma_G\tau\}<1$, it follows that
$G(t)\longrightarrow0$. Therefore, solutions $T_S$ and $T_{NS}$ of
(\ref{fa})-(\ref{ifa}) satisfy system (\ref{tstns}). Since
$\mathcal{R}_T^0<1$, one has $(T_S,~ T_{NS})\longrightarrow
(\bar{T}_S,~ \bar{T}_{NS})$. Point (iii) of Theorem \ref{prop}
holds.

\section*{Appendix C: Proof of Theorem \ref{sol_coexist}}
Taking new variables
$T_S(t)=e^{x(t)},~T_{NS}(t)=e^{y(t)},~G(t)=e^{z(t)}$ then system
$(\ref{fa})-(\ref{ifa})$ becomes,
\begin{equation}\label{change}
    \left\{
\begin{array}{lcl}
  \dot{x}(t) & = & -\omega_S-\mu_S-\sigma_Ge^z+(\gamma_S+\gamma_{NS}e^ye^{-x})\left(1-\displaystyle\frac{e^x+e^y}{K_T}\right),~~~t\neq t_n. \\
  \dot{y}(t) & = & -\mu_{NS}+\omega_Se^xe^{-y},~~~ t_{n+1}=t_n+\tau, \\
  \dot{z}(t) & = & \gamma_G\left(1-\displaystyle\frac{e^z}{K_G}\right)-\sigma_{NS}e^y-\mu_G, \\
  x(t^+) & = & x(t)+\ln(1-\eta_Sw(e^z)), ~~~t=t_n.\\
  y(t^+) & = & y(t), ~~n=0,1,2,...,\\
  z(t^+) & = & z(t)+\ln(1-\eta_G).
\end{array}
    \right.
\end{equation}

Let $X=C^1([0, \tau], \mathbb{R}^3)$, $Z=C^1([0, \tau],
\mathbb{R}^3)\times C^1([0, \tau], \mathbb{R}^3)$ and for $u=(x, y,
z)\in X$,
$$||u||=\max\limits_{t\in[0, \tau]}|x(t)|+\max\limits_{t\in[0,
\tau]}|y(t)|+\max\limits_{t\in[0, \tau]}|z(t)|.$$ Then $X$, $Z$ are
Banach spaces when they are endowed with the above norm $||\cdot||$.
\par Let, $$L: Dom(L)\subset X \rightarrow Z, \left(
                                      \begin{array}{c}
                                        x \\
                                        y \\
                                        z \\
                                      \end{array}
                                    \right)\rightarrow \left(\left(
                                                               \begin{array}{c}
                                                                 \dot{x} \\
                                                                 \dot{y} \\
                                                                 \dot{z} \\
                                                               \end{array}
                                                             \right)
                                    ,~\left(
                                        \begin{array}{c}
                                          \Delta x(t_n) \\
                                          \Delta y(t_n) \\
                                          \Delta z(t_n) \\
                                        \end{array}
                                      \right)
                                    \right)
$$ and $$N\left(
                                      \begin{array}{c}
                                        x \\
                                        y \\
                                        z \\
                                      \end{array}
                                    \right)= \left(N_1\left(
                                      \begin{array}{c}
                                        x \\
                                        y \\
                                        z \\
                                      \end{array}
                                    \right),N_2\left(
                                      \begin{array}{c}
                                        x \\
                                        y \\
                                        z \\
                                      \end{array}
                                    \right)\right),$$ where

$$N_1\left(
                                      \begin{array}{c}
                                        x \\
                                        y \\
                                        z \\
                                      \end{array}
                                    \right)=\left(
                                                               \begin{array}{c}
                                                                 -\omega_S-\mu_S-\sigma_Ge^z+(\gamma_S+\gamma_{NS}e^ye^{-x})\left(1-\displaystyle\frac{e^x+e^y}{K_T}\right) \\
                                                                 -\mu_{NS}+\omega_Se^xe^{-y} \\
                                                                 \gamma_G\left(1-\displaystyle\frac{e^z}{K_G}\right)-\sigma_{NS}e^y-\mu_G \\
                                                               \end{array}
                                                             \right),$$
$$N_2\left(
                                      \begin{array}{c}
                                        x \\
                                        y \\
                                        z \\
                                      \end{array}
                                    \right)=\left(
                                        \begin{array}{c}
                                          \ln(1-\eta_Sw(e^{z(\tau)})) \\
                                          0 \\
                                          \ln(1-\eta_G) \\
                                        \end{array}
                                      \right).$$

A direct computation leads to
$$KerL=\left\{\left(
                                      \begin{array}{c}
                                        x \\
                                        y \\
                                        z \\
                                      \end{array}
                                    \right): \left(
                                      \begin{array}{c}
                                        x \\
                                        y \\
                                        z \\
                                      \end{array}
                                    \right)=\left(
                                      \begin{array}{c}
                                        c_1 \\
                                        c_2 \\
                                        c_3 \\
                                      \end{array}
                                    \right)\in\mathbb{R}^3,~t\in[0,
                                    \tau]\right\}$$ and
$$ImL=\left\{\left(\left(
                                      \begin{array}{c}
                                        l \\
                                        m \\
                                        n \\
                                      \end{array}
                                    \right),~\left(
                                      \begin{array}{c}
                                        a \\
                                        b \\
                                        c \\
                                      \end{array}
                                    \right)\right)\in Z: \left(
                                      \begin{array}{c}
                                        \int_0^\tau l(t)dt+a=0 \\
                                        \int_0^\tau m(t)dt+b=0 \\
                                        \int_0^\tau n(t)dt+c=0 \\
                                      \end{array}
                                    \right)\right\}.$$
Since $ImL$ is closed in $Z$, $L$ is a Fredholm mapping of index
zero.
Indeed,
$$Index(L)=dim(KerL)-dim(CoKerL)=dim(KerL)-(dim(Z)-dim(ImL))=3-(6-3)=0.$$
Thus following (Gaines and Mawhin (1977) \cite{Gaines1977}, Page
12), there exist two continuous projectors $P$
 and $Q$ such that the sequel $X
 \stackrel{P}{\longrightarrow}DomL\stackrel{L}{\longrightarrow}Z\stackrel{Q}{\longrightarrow}Z$
is exact i.e $ImP=KerL$ and
 $KerQ=ImL=Im(I-Q)$. It suffices to choose
 $$P\left(
                                      \begin{array}{c}
                                        x \\
                                        y \\
                                        z \\
                                      \end{array}
                                    \right)=\left(
                                      \begin{array}{c}
                                        x(\tau) \\
                                        y(\tau) \\
                                        z(\tau) \\
                                      \end{array}
                                    \right)~and ~Q\left(\left(
                                      \begin{array}{c}
                                        l \\
                                        m \\
                                        n \\
                                      \end{array}
                                    \right),~\left(
                                      \begin{array}{c}
                                        a \\
                                        b \\
                                        c \\
                                      \end{array}
                                    \right)\right)=\left(\displaystyle\frac{1}{\tau}\left(
                                                                       \begin{array}{c}
                                                                         \int_0^\tau l(s)dt+a \\
                                                                         \int_0^\tau m(s)dt+b \\
                                                                         \int_0^\tau n(s)dt+c \\
                                                                       \end{array}
                                                                     \right)
                                    ,~\left(
                                        \begin{array}{c}
                                          0 \\
                                          0 \\
                                          0 \\
                                        \end{array}
                                      \right)
                                    \right).$$
 One can verify that
                                    $LP\left(
                                      \begin{array}{c}
                                        x \\
                                        y \\
                                        z \\
                                      \end{array}
                                    \right)=0_{X}$ and $QL\left(
                                      \begin{array}{c}
                                        x \\
                                        y \\
                                        z \\
                                      \end{array}
                                    \right)=0_{Z}.$\\
 Furthermore, the generalized inverse $K_{P}: ImL\rightarrow KerP\cap
 Dom(L)$ of the map $L: KerP\cap
 Dom(L)\rightarrow ImL$ is given by
 $$K_{P}\left(\left(
                                      \begin{array}{c}
                                        l \\
                                        m \\
                                        n \\
                                      \end{array}
                                    \right),~\left(
                                      \begin{array}{c}
                                        a \\
                                        b \\
                                        c \\
                                      \end{array}
                                    \right)\right)=\left(
          \begin{array}{c}
            \int_0^tl(s)ds+a \\
            \int_0^tm(s)ds+b \\
            \int_0^\tau n(s)ds+c \\
          \end{array}
        \right).
 $$
 Indeed, let $u=(u_1, u_2, u_3)^T\in KerP\cap
 Dom(L),~(g,r)=((g_1,g_2,g_3),(r_1,r_2,r_3))\in ImL$,  we have
 \begin{equation}
    \begin{array}{ccl}
      K_{P}L(u(t)) & = & K_{P}(\dot{u},~\Delta u) \\
       & = & \int_0^t\dot{u}(s)ds+\Delta u \\
       & = & u(t)-u(0)+u(0)-u(\tau) \\
       & = & u(t)-P(u) \\
       & = & u(t),~ because ~u\in KerP
    \end{array}
 \end{equation}
and
\begin{equation}
    \begin{array}{ccl}
      LK_{P}(g(t),r) & = & L(\int_0^tg(s)ds+r) \\
       & = & (g(t),~-\int_0^\tau g(t)dt) \\
       & = & (g(t),~r)~ because ~(g,r)\in ImL. \\
    \end{array}
\end{equation}

 Thus,
 $$QN\left(
                                      \begin{array}{c}
                                        x \\
                                        y \\
                                        z \\
                                      \end{array}
                                    \right)=\left(\left(
                                                    \begin{array}{c}
                                                      A_1 \\
                                                      A_2 \\
                                                      A_3 \\
                                                    \end{array}
                                                  \right)
                                    ,~\left(
                                        \begin{array}{c}
                                          0 \\
                                          0 \\
                                          0 \\
                                        \end{array}
                                      \right)
                                    \right).$$
Furthermore,

\begin{equation}
    \begin{array}{ccl}
      K_{P}(I-Q)N\left(
                                      \begin{array}{c}
                                        x \\
                                        y \\
                                        z \\
                                      \end{array}
                                    \right) & = & K_{P}N\left(
                                      \begin{array}{c}
                                        x \\
                                        y \\
                                        z \\
                                      \end{array}
                                    \right)-K_{P}QN\left(
                                      \begin{array}{c}
                                        x \\
                                        y \\
                                        z \\
                                      \end{array}
                                    \right) \\
       & = & \left(
                                              \begin{array}{c}
                                                B_1 \\
                                                B_2 \\
                                                B_3 \\
                                              \end{array}
                                            \right)-\left(
                                                      \begin{array}{c}
                                                        C_1 \\
                                                        C_2 \\
                                                        C_3 \\
                                                      \end{array}
                                                    \right)+\left(
                                                              \begin{array}{c}
                                                                D_1 \\
                                                                D_2 \\
                                                                D_3 \\
                                                              \end{array}
                                                            \right),
    \end{array}
\end{equation}

 where
 \begin{equation}
    \begin{array}{lll}
      A_1 & = &  \gamma_S-\omega_S-\mu_S-\displaystyle\frac{1}{\tau}\displaystyle\int_0^\tau(\sigma_Ge^{z(t)}+\displaystyle\frac{\gamma_S}{K_T}(e^{x(t)}+e^{y(t)}))dt\\
      &&+\displaystyle\frac{1}{\tau}\int_0^\tau\gamma_{NS}e^{-x(t)}e^{y(t)}\left(1-\displaystyle\frac{e^{x(t)}+e^{y(t)}}{K_T}\right)dt+\displaystyle\frac{1}{\tau}\ln(1-\eta_Sw(e^{z(\tau)})),\\
      A_2 & = &  -\mu_{NS}+\displaystyle\frac{1}{\tau}\displaystyle\int_0^\tau\omega_Se^{x(t)}e^{-y(t)}dt,\\
      A_3 & = &  \gamma_G-\mu_G-\displaystyle\frac{1}{\tau}\displaystyle\int_0^\tau\displaystyle\frac{\gamma_G}{K_G}e^{z(t)}dt-\displaystyle\frac{1}{\tau}\displaystyle\int_0^\tau\sigma_{NS}e^{y(t)}dt+\displaystyle\frac{1}{\tau}\ln(1-\eta_G),\\
\end{array}
 \end{equation}
\begin{equation}
    \begin{array}{lll}
      B_1 & = & \displaystyle\int_0^t\left(-\omega_S-\mu_S-\sigma_Ge^{z(s)}+(\gamma_S+\gamma_{NS}e^{y(s)}e^{-x(s)})\left(1-\displaystyle\frac{e^{x(s)}+e^{y(s)}}{K_T}\right)\right)ds,\\
      B_2 & = & \displaystyle\int_0^t\left(-\mu_{NS}+\omega_Se^{x(s)}e^{-y(s)}\right)ds, \\
      B_3 & = & \displaystyle\int_0^t\left(\gamma_G\left(1-\displaystyle\frac{e^{z(s)}}{K_G}\right)-\sigma_{NS}e^{y(s)}-\mu_G\right)ds, \\
\end{array}
 \end{equation}
 \begin{equation}
    \begin{array}{lll}
      C_1 & = & \displaystyle\frac{t}{\tau}\left(\displaystyle\int_0^\tau\left(-\omega_S-\mu_S-\sigma_Ge^{z(s)}+(\gamma_S+\gamma_{NS}e^{y(s)}e^{-x(s)})\left(1-\displaystyle\frac{e^{x(s)}+e^{y(s)}}{K_T}\right)\right)ds\right.\\
      &&\left.+\ln(1-\eta_Sw(e^{z(\tau)}))\right),\\
      C_2 & = & \displaystyle\frac{t}{\tau}\displaystyle\int_0^\tau\left(-\mu_{NS}+\omega_Se^{x(s)}e^{-y(s)}\right)ds, \\
      C_3 & = & \displaystyle\frac{t}{\tau}\left(\displaystyle\int_0^\tau\left(\gamma_G\left(1-\displaystyle\frac{e^{z(s)}}{K_G}\right)-\sigma_{NS}e^{y(s)}-\mu_G\right)ds+\ln(1-\eta_G)\right), \\
      &&\\
      D_1 & = & \ln(1-\eta_Sw(e^{z(\tau)})),\\
      D_2 & = & 0,\\
      D_3 & = & \ln(1-\eta_G). \\
     \end{array}
 \end{equation}

 Clearly, $QN$ and $K_{P}(I-Q)N$ are continuous then for any open bounded set $\Omega\subset X$, $QN(\bar{\Omega})$ is
 bounded. Furthermore, let $t_1,~t_2\in [0,~\tau]$, $u(t)=(x, y,
 z)(t)$, $$f(t,u(t))=\left(
                                                               \begin{array}{c}
                                                                 -\omega_S-\mu_S-\sigma_Ge^z+(\gamma_S+\gamma_{NS}e^ye^{-x})\left(1-\displaystyle\frac{e^x+e^y}{K_T}\right) \\
                                                                 -\mu_{NS}+\omega_Se^xe^{-y} \\
                                                                 \gamma_G\left(1-\displaystyle\frac{e^z}{K_G}\right)-\sigma_{NS}e^y-\mu_G \\
                                                               \end{array}
                                                             \right)$$
 and
 $$a=\left(
                                        \begin{array}{c}
                                          \ln(1-\eta_Sw(e^{z(\tau)})) \\
                                          0 \\
                                          \ln(1-\eta_G) \\
                                        \end{array}
                                      \right).$$
 We have\\
 $|K_{P}(I-Q)N(u(t_2))-K_{P}(I-Q)N(u(t_1))|$
\begin{equation}\label{ekicon}
    \begin{array}{ccl}
       & = & \left|\displaystyle\int_0^{t_2}f(s,u(s))ds-\displaystyle\int_0^{t_1}f(s,u(s))ds\right. \left.-\displaystyle\frac{t_2}{\tau}\left\{\displaystyle\int_0^{\tau}f(s,u(s))ds+a\right\}\right.\\
       &&\left.+\displaystyle\frac{t_1}{\tau}\left\{\displaystyle\int_0^{\tau}f(s,u(s))ds+a\right\}\right|\\
       & = & \left|\displaystyle\int_{t_1}^{t_2}f(s,u(s))ds-\displaystyle\frac{(t_2-t_1)}{\tau}\left\{\displaystyle\int_0^{\tau}f(s,u(s))ds+a\right\}\right| \\
       & \leq &
       |t_2-t_1|\max\limits_{t\in[0,~\tau]}|f(t,u(t))|+\displaystyle\frac{|t_2-t_1|}{\tau}\left(\tau\max\limits_{t\in[0,~\tau]}|f(t,u(t))|+a\right)
       \\
       &\leq & |t_2-t_1|\left(2\max\limits_{t\in[0,~\tau]}|f(t,u(t))|+\displaystyle\frac{a}{\tau}\right)
    \end{array}
\end{equation}
and
\begin{equation}\label{born}
    \begin{array}{ccl}
      |K_{P}(I-Q)N(u(t))| & \leq & |a|+\tau\max\limits_{t\in[0,~\tau]}|f(t,u(t))|+|a|+\tau\max\limits_{t\in[0,~\tau]}|f(t,u(t))| \\
       & \leq &
       2\left(|a|+\tau\max\limits_{t\in[0,~\tau]}|f(t,u(t))|\right).
    \end{array}
\end{equation}
Then using relations (\ref{ekicon}), (\ref{born}) and the
 Arzela-Ascoli's theorem (Sonntag (1997) \cite{Sonntag1997} Theorem 3.1, Page 314) we deduce that\\ $K_{P}(I-Q)N(\bar{\Omega})$ is
 compact. Thus, $N$ is a $L-$compact mapping on $\bar{\Omega}$.  The isomorphism
 $J$ of $ImQ$ onto $KerL$ may be defined by
 $$J: ImQ\rightarrow X,~\left(\left(
                                      \begin{array}{c}
                                        u \\
                                       v \\
                                        w \\
                                      \end{array}
                                    \right),~\left(
                                      \begin{array}{c}
                                        0 \\
                                        0 \\
                                        0 \\
                                      \end{array}
                                    \right)\right)\rightarrow \left(
                                      \begin{array}{c}
                                        u \\
                                        v \\
                                        w \\
                                     \end{array}
                                    \right).$$
 Now we reach the position to search for an appropriate open,
 bounded subset $\Omega$ for the application of the continuation
 theorem, i.e we search $M_0$ such that every $\tau$-periodic solution
 of system $(\ref{fa})-(\ref{ifa})$ satisfied $|x(t)|+|y(t)|+|z(t)|\leq
 M_0$ with $0\leq t\leq \tau.$\\ Corresponding to the operator equation
 $Lx=\beta Nx$, $\beta\in(0,1)$, we have
\begin{equation}\label{rechange}
    \left\{
\begin{array}{lcl}
  \dot{x}(t) & = & \beta\left[-\omega_S-\mu_S-\sigma_Ge^z+(\gamma_S+\gamma_{NS}e^ye^{-x})\left(1-\displaystyle\frac{e^x+e^y}{K_T}\right)\right],~~~t\neq t_n. \\
  \dot{y}(t) & = & \beta\left[-\mu_{NS}+\omega_Se^xe^{-y}\right],~~~ t_{n+1}=t_n+\tau, \\
  \dot{z}(t) & = & \beta\left[\gamma_G\left(1-\displaystyle\frac{e^z}{K_G}\right)-\sigma_{NS}e^y-\mu_G\right], \\
  x(t^+)-x(t) & = & \beta\ln(1-\eta_Sw(e^z)), ~~~t=t_n.\\
  y(t^+)-y(t) & = & 0, ~~n=0,1,2,...,\\
  z(t^+)-z(t) & = & \beta\ln(1-\eta_G).
\end{array}
    \right.
\end{equation}
Suppose that $(x(t), y(t), z(t))\in X$ is an arbitrary solution of
system (\ref{rechange}) for a certain $\beta\in(0,1)$. Integrating
on both sides of $(\ref{rechange})$ over the interval $[0,~ \tau]$,
we obtain
\begin{equation}\label{integrale}
    \left\{
\begin{array}{rcl}
  \displaystyle\int_0^\tau\left[-\displaystyle\frac{\gamma_S}{K_T}(e^x+e^y)+\gamma_{NS}e^ye^{-x}\left(1-\displaystyle\frac{e^x+e^y}{K_T}\right)-\sigma_Ge^z\right]dt & = & (\omega_S+\mu_S-\gamma_S)\tau\\
  &&-\ln(1-\eta_Sw(e^{z(\tau)})), \\
  \displaystyle\int_0^\tau\omega_Se^xe^{-y}dt & = & \mu_{NS}\tau, \\
  \displaystyle\int_0^\tau\left[\displaystyle\frac{\gamma_G}{K_G}e^z+\sigma_{NS}e^y\right]dt&=&(\gamma_G-\mu_G)\tau+\ln(1-\eta_G). \\
  \end{array}
    \right.
\end{equation}
Note that assumptions of Theorem \ref{sol_coexist} lead
$$(\gamma_G-\mu_G)\tau+\ln(1-\eta_G)>0.$$
Since $X$ is a Banach space
and $(x(t), y(t), z(t))\in X$, there exist
$\overline{\xi},~\underline{\xi},~\overline{\eta},~\underline{\eta},~\overline{\tau}$
and $\underline{\tau}$ such that
\begin{equation}
    \begin{array}{cc}
      x(\overline{\xi})=\max\limits_{0\leq t\leq \tau}x(t), & x(\underline{\xi})=\min\limits_{0\leq t\leq \tau}x(t), \\
      y(\overline{\eta})=\max\limits_{0\leq t\leq \tau}y(t), & y(\underline{\eta})=\min\limits_{0\leq t\leq \tau}y(t), \\
      z(\overline{\tau})=\max\limits_{0\leq t\leq \tau}z(t), & z(\underline{\tau})=\min\limits_{0\leq t\leq
      \tau}z(t).
    \end{array}
\end{equation}
It follows from system $(\ref{integrale})$ that
\begin{equation}
\begin{array}{lll}
  \displaystyle\int_0^\tau|\dot{x}(t)|dt & \leq & (\omega_S+\mu_S)\tau+\displaystyle\int_0^\tau\left|-\sigma_Ge^z+(\gamma_S+\gamma_{NS}e^ye^{-x})\left(1-\displaystyle\frac{e^x+e^y}{K_T}\right)\right|dt \\
   & \leq & (\omega_S+\mu_S)\tau+\sigma_G\tau e^{z(\overline{\tau})}+\displaystyle\int_0^\tau(\gamma_S+\gamma_{NS}e^ye^{-x})dt \\
   & \leq &
   (\omega_S+\mu_S)\tau+\sigma_G\tau e^{z(\overline{\tau})}+\displaystyle\int_0^\tau(\gamma_S+\gamma_{NS}e^y)dt\\
   & \leq &
   (\omega_S+\mu_S+\gamma_S)\tau+\sigma_G\tau e^{z(\overline{\tau})}+\gamma_{NS}\tau e^{y(\overline{\eta})},
\end{array}
\end{equation}
\begin{equation}
\begin{array}{lll}
  \displaystyle\int_0^\tau|\dot{y}(t)|dt & \leq & \mu_{NS}\tau+\displaystyle\int_0^\tau|\omega_Se^xe^{-y}|dt \\
   & \leq & 2\mu_{NS}\tau
   \end{array}
\end{equation}
and
\begin{equation}
\begin{array}{lll}
  \displaystyle\int_0^\tau|\dot{z}(t)|dt & \leq & (\gamma_G+\mu_G)\tau+\displaystyle\int_0^\tau\left(\displaystyle\frac{\gamma_G}{K_G}e^z+\sigma_{NS}e^y\right)dt \\
   & \leq & (\gamma_G+\mu_G)\tau+(\gamma_G-\mu_G)\tau+\ln(1-\eta_G)\\
   & \leq & 2\gamma_G\tau+\ln(1-\eta_G).
   \end{array}
\end{equation}
Recall that
$$
\begin{array}{ccl}
  \rho_G^0>1 & \Longleftrightarrow & (\gamma_G-\mu_G)\tau+\ln(1-\eta_G)>0 \\
   & \Longrightarrow & 2\gamma_G\tau+\ln(1-\eta_G)>0.
\end{array}
$$
 Since
\begin{equation}
\begin{array}{lll}
  \tau\left(\displaystyle\frac{\gamma_G}{K_G}e^{z(\underline{\tau})}+\sigma_{NS}e^{y(\underline{\eta})}\right) & \leq & (\gamma_G-\mu_G)\tau-\displaystyle\int_0^\tau\dot{z}(t)dt \\
   & \leq & (\gamma_G-\mu_G)\tau+\ln(1-\eta_G),
   \end{array}
\end{equation}
then
\begin{equation}
\begin{array}{lll}
  z(\underline{\tau}) & \leq & \ln\left\{\displaystyle\frac{K_G}{\gamma_G}\left((\gamma_G-\mu_G)+\displaystyle\frac{\ln(1-\eta_G)}{\tau}\right)\right\} \\
  &&\\
  y(\underline{\eta}) & \leq &
  \ln\left\{\displaystyle\frac{1}{\sigma_{NS}}\left((\gamma_G-\mu_G)+\displaystyle\frac{\ln(1-\eta_G)}{\tau}\right)\right\}.
   \end{array}
\end{equation}
Moreover, from
$$\displaystyle\int_0^\tau\omega_Se^xe^{-y}dt=\mu_{NS}\tau$$
 we deduce
$$x(\underline{\xi})\leq\ln\left(\displaystyle\frac{\mu_{NS}}{\omega_S}e^{y(\overline{\eta})}\right)\, and\,x(\overline{\xi})\geq\ln\left(\displaystyle\frac{\mu_{NS}}{\omega_S}\right).$$
Furthermore,
$$\int_0^\tau\dot{z}(t)dt=\int_0^\tau\left[\gamma_G-\mu_G-\displaystyle\frac{\gamma_G}{K_G}e^{z(t)}-\sigma_{NS}e^{y(t)}\right]=-\ln(1-\eta_G).$$
Using
\begin{equation}
\begin{array}{ccl}
 0<(\gamma_G-\mu_G)\tau+\ln(1-\eta_G)  & = & \displaystyle\int_0^\tau\left[\displaystyle\frac{\gamma_G}{K_G}e^{z(t)}+\sigma_{NS}e^{y(t)}\right]dt \\
 &&\\
   & \leq & \displaystyle\int_0^\tau\left[\displaystyle\frac{\gamma_G}{K_G}e^{z(\overline{\tau})}+\sigma_{NS}e^{y(\overline{\eta})}\right]dt=\tau\left[\displaystyle\frac{\gamma_G}{K_G}e^{z(\overline{\tau})}+\sigma_{NS}e^{y(\overline{\eta})}\right],
\end{array}
\end{equation}
then there exist $\zeta_1>0,~\zeta_2>0$ such that
\begin{itemize}
    \item[$\bullet$]$\zeta_1+\zeta_2=(\gamma_G-\mu_G)+\displaystyle\frac{\ln(1-\eta_G)}{\tau}$,
    \item[$\bullet$] $\displaystyle\frac{\gamma_G}{K_G}e^{z(\overline{\tau})}\geq\zeta_1$ and
    \item[$\bullet$] $\sigma_{NS}e^{y(\overline{\eta})}\geq\zeta_2.$
\end{itemize}
Thus,
\begin{equation}
\begin{array}{llllll}
  z(\overline{\tau}) & \geq & \ln\left\{\displaystyle\frac{K_G}{\gamma_G}\zeta_1\right\}, &y(\overline{\eta}) & \geq &
  \ln\left\{\displaystyle\frac{1}{\sigma_{NS}}\zeta_2\right\}.
   \end{array}
\end{equation}
So, keeping in mind that by assumptions of Theorem
\ref{sol_coexist}, one has
$$(\gamma_G-\mu_G)+\displaystyle\frac{\ln(1-\eta_G)}{\tau}>0,$$
\begin{equation}
    \begin{array}{lllr}
      x(t) & \leq & x(\underline{\xi})+\int_0^\tau|\dot{x}(t)|dt& \\
       & \leq & \ln\left(\displaystyle\frac{\mu_{NS}}{\omega_S}e^{y(\overline{\eta})}\right)+ (\omega_S+\mu_S+\gamma_S)\tau+\sigma_G\tau e^{z(\overline{\tau})}+\gamma_{NS}\tau e^{y(\overline{\eta})}&:=x_u,\\
      y(t) & \leq & y(\underline{\eta})+\int_0^\tau|\dot{y}(t)|dt&\\
       & \leq & \ln\left\{\displaystyle\frac{1}{\sigma_{NS}}\left((\gamma_G-\mu_G)+\displaystyle\frac{\ln(1-\eta_G)}{\tau}\right)\right\}+2\mu_{NS}\tau&:=y_u, \\
      z(t) & \leq & z(\underline{\tau})+\int_0^\tau|\dot{z}(t)|dt& \\
       & \leq
       &\ln\left\{\displaystyle\frac{K_G}{\gamma_G}\left((\gamma_G-\mu_G)+\displaystyle\frac{\ln(1-\eta_G)}{\tau}\right)\right\}+2\gamma_G\tau+\ln(1-\eta_G)&:=z_u,\\
       x(t) & \geq & x(\overline{\xi})-\int_0^\tau|\dot{x}(t)|dt& \\
       & \geq & \ln\left(\displaystyle\frac{\mu_{NS}}{\omega_S}\right)- (\omega_S+\mu_S+\gamma_S)\tau-\sigma_G\tau e^{z(\overline{\tau})}-\gamma_{NS}\tau e^{y(\overline{\eta})}&:=x_l,\\
      y(t) & \geq & y(\overline{\eta})-\int_0^\tau|\dot{y}(t)|dt&\\
       & \geq & \ln\left\{\displaystyle\frac{\zeta_2}{\sigma_{NS}}\right\}-2\mu_{NS}\tau&:=y_l, \\
      z(t) & \geq & z(\overline{\tau})-\int_0^\tau|\dot{z}(t)|dt& \\
       & \geq
       &\ln\left\{\displaystyle\frac{K_G}{\gamma_G}\zeta_1\right\}-2\gamma_G\tau-\ln(1-\eta_G)&:=z_l,
    \end{array}
\end{equation}
therefore we obtain,
\begin{equation}
    \begin{array}{ccc}
      \max\limits_{0\leq t\leq \tau}|x(t)| & \leq & \max\left\{|x_u|, |x_l|\right\}:=M_x, \\
      \max\limits_{0\leq t\leq \tau}|y(t)| & \leq & \max\left\{|y_u|, |y_l|\right\}:=M_y, \\
      \max\limits_{0\leq t\leq \tau}|z(t)| & \leq & \max\left\{|z_u|,
      |z_l|\right\}:=M_z.
    \end{array}
\end{equation}
$M_x$, $M_y$ and $M_z$ are independent of $\beta$. \par Now let us
consider the algebraic equations
\begin{equation}\label{algebric}
    \left\{
\begin{array}{rcl}
  \gamma_S-\omega_S-\mu_S+\displaystyle\frac{1}{\tau}\ln(1-\eta_Sw(e^{z(\tau)}))&~~&~~~~\\
  -\displaystyle\frac{1}{\tau}\displaystyle\int_0^\tau\left[\displaystyle\frac{\gamma_S}{K_T}(e^x+e^y)-\beta\gamma_{NS}e^ye^{-x}\left(1-\displaystyle\frac{e^x+e^y}{K_T}\right)+\beta\sigma_Ge^z\right]dt & = & 0, \\
  -\mu_{NS}+\displaystyle\frac{1}{\tau}\int_0^\tau\omega_Se^xe^{-y}dt & = & 0, \\
  (\gamma_G-\mu_G)+\displaystyle\frac{1}{\tau}\ln(1-\eta_G)-\displaystyle\frac{1}{\tau}\displaystyle\int_0^\tau\left[\displaystyle\frac{\gamma_G}{K_G}e^z+\sigma_{NS}e^y\right]dt&=&0
  \end{array}
    \right.
\end{equation}
for $(x,y,z)\in\mathbb{R}^3$, where $\beta\in[0,~1]$ is a parameter.
By carrying out similar arguments as in system (\ref{integrale}),
one can show that any solution $(x^*, y^*, z^*)$ of
$(\ref{algebric})$ with $\beta\in[0,~1]$ satisfies
\begin{equation}\label{ine}
l_1\leq x^*\leq L_1,~~l_2\leq y^*\leq L_2,~~l_3\leq z^*\leq L_3.
\end{equation}
Taking $M_0=M_x+M_y+M_z+M_k$ where $M_k$>0 is taken sufficiently
large such that\\ $M_k>|l_1|+|L_1|+|l_2|+|L_2|+|l_3|+|L_3|$,
 we define $\Omega=\left\{(x,y,z)^T\in X:
||(x,y,z)||<M_0\right\}$, then $\Omega$ verifies the requirement
$(1)$ of The Continuation Theorem (Gaines and Mahwin (1977)
\cite{Gaines1977}, Page 40). When $(x,y,z)\in\partial\Omega\cap
KerL=\partial\Omega\cap\mathbb{R}^3$, $(x,y,z)$ is a constant vector
in $\mathbb{R}^3$ with $||(x,y,z)||=M_0$. Then from $(\ref{ine})$
and the definition of $M_0$, one has
\begin{equation}\label{q}
QN\left(
                                      \begin{array}{c}
                                        x \\
                                        y \\
                                        z \\
                                      \end{array}
                                    \right)=\left(\left(
                                                    \begin{array}{c}
                                                      A^{(1)} \\
                                                      A^{(2)} \\
                                                      A^{(3)} \\
                                                    \end{array}
                                                  \right)
                                    ,~\left(
                                        \begin{array}{c}
                                          0 \\
                                          0 \\
                                          0 \\
                                        \end{array}
                                      \right)
                                    \right)\neq\left(\left(
                                                    \begin{array}{c}
                                                      0 \\
                                                      0 \\
                                                      0 \\
                                                    \end{array}
                                                  \right)
                                    ,~\left(
                                        \begin{array}{c}
                                          0 \\
                                          0 \\
                                          0 \\
                                        \end{array}
                                      \right)
                                    \right),
\end{equation}
where,
\begin{equation}
    \begin{array}{lll}
      A^{(1)} & = &  \gamma_S-\omega_S-\mu_S-\displaystyle\frac{1}{\tau}\displaystyle\int_0^\tau\left(\sigma_Ge^{z}+\displaystyle\frac{\gamma_S}{K_T}(e^{x}+e^{y})\right)dt+\displaystyle\frac{1}{\tau}\displaystyle\int_0^\tau\gamma_{NS}e^{-x}e^{y}\left(1-\displaystyle\frac{e^{x}+e^{y}}{K_T}\right)dt\\
      &&+\displaystyle\frac{1}{\tau}\ln(1-\eta_Sw(e^{z(\tau)})),\\
      A^{(2)} & = &  -\mu_{NS}+\displaystyle\frac{1}{\tau}\displaystyle\int_0^\tau\omega_Se^{x}e^{-y}dt,\\
      A^{(3)} & = &  \gamma_G-\mu_G-\displaystyle\frac{1}{\tau}\displaystyle\int_0^\tau\displaystyle\frac{\gamma_G}{K_G}e^{z}dt-\displaystyle\frac{1}{\tau}\displaystyle\int_0^\tau\sigma_{NS}e^{y}dt+\displaystyle\frac{1}{\tau}\ln(1-\eta_G),
      \end{array}
 \end{equation}
that is, the first part of $(2)$ of The Continuation Theorem (Gaines
and Mahwin (1977) \cite{Gaines1977}, Page 40) is valid.
\par
To compute the Brouwer degree, let us consider the homotopy
$$H_\beta((x,y,z)^T)=\beta JQN((x,y,z)^T)+(1-\beta)V((x,y,z)^T),~~\beta\in[0,~1],$$
where
$$V((x,y,z)^T)=\left(
                 \begin{array}{c}
                   \gamma_S-\omega_S-\mu_S+\displaystyle\frac{1}{\tau}\ln\left(1-\eta_Sw(e^{z(\tau)})\right)-\displaystyle\frac{\gamma_S}{K_T}(e^x+e^y) \\
                   -\mu_{NS}+\omega_Se^xe^{-y} \\
                   (\gamma_G-\mu_G)+\displaystyle\frac{1}{\tau}\ln(1-\eta_G)-\displaystyle\frac{\gamma_G}{K_G}e^z-\sigma_{NS}e^y \\
                 \end{array}
               \right).
$$
From $(\ref{algebric})$, it follows that $0\notin
H_\beta(\partial\Omega\cap KerL)$ for $\beta\in[0,~1]$. Moreover,
since
$-\displaystyle\frac{\gamma_S\gamma_G(\omega_S+\mu_{NS})}{K_TK_G}\neq0$,
the algebraic equation $V((x,y,z)^T)=0$ has a unique solution
$(e^{x^*},e^{y^*},e^{z^*})^T\in\mathbb{R}^3$. We compute the Brouwer
degree ($deg(\cdot,\cdot,\cdot)$) by using the invariance property
of homotopy \cite{Fan2004}, one has
\begin{equation}
    \begin{array}{ccl}
      deg(JQN, \Omega\cap KerL,0) & = & deg(V, \Omega\cap
KerL,0) \\
       & = &  \sum\limits_{p\in
V^{-1}(0)}sign(J_V(p))\\
       & = &sign\left[\det\left(                                                                                                                \begin{array}{ccc}
                                                                                                                  -\displaystyle\frac{\gamma_S}{K_T}e^{x^*} & -\displaystyle\frac{\gamma_S}{K_T}e^{y^*} & 0 \\
                                                                                                                  \mu_{NS} & -\mu_{NS} & 0 \\
                                                                                                                  0 & -\sigma_{NS}e^{y^*} & -\displaystyle\frac{\gamma_G}{K_G}e^{z^*} \\
                                                                                                                \end{array}
                                                                                                              \right)\right]\\
                                                                                                              &=&sign\left[-\displaystyle\frac{\gamma_G\gamma_S}{K_TK_G}(\mu_{NS}+\omega_S)e^{x^*}e^{z^*}\right],\,since\,\omega_Se^{x^*}=\mu_{NS}e^{y^*}\\
&= & -1 \neq  0.
    \end{array}
\end{equation}
 By now, we
have prove that $\Omega$ verifies all requirements of The
Continuation Theorem (Gaines and Mahwin (1977) \cite{Gaines1977},
Page 40), then $Lx=Nx$ has at least one solution in
$DomL\cap\overline{\Omega}$, i.e. system $(\ref{change})$ has at
least one solution in $DomL\cap\overline{\Omega}$, say
$(x^*(t),y^*(t),z^*(t))^T$. Set
$T_S^*(t)=\exp(x^*(t)),~T_{NS}^*(t)=\exp(y^*(t),~G^*(t)=\exp(z^*(t))$,
then $(T_S^*(t),T_{NS}^*(t),G^*(t))^T$ is a positive and
$\tau$-periodic solution of system $(\ref{fa})-(\ref{ifa})$. This
completes the proof.

\section*{Appendix D: Particular values of $\sigma_{NS}$, $\sigma_{G}$ and $\tau$}
With respect to relations (\ref{SG}) and (\ref{SF}), we set
\begin{equation}\label{TGF}
\left\{
\begin{array}{rcl}
\sigma_G^\star(\tau)&=& \displaystyle\frac{1}{G_{int}}\left(\gamma_S-(\mu_S+\omega_S+\mu_{NS})\right),\\
&&\\
&=& \displaystyle\frac{(\gamma_S-(\mu_S+\omega_S+\mu_{NS}))}{\displaystyle\frac{K_G}{\gamma_G}\left(\gamma_G-\mu_G+\displaystyle\frac{\ln(1-\eta_G)}{\tau}\right)},\\
&&\\
 \sigma_{NS}^\star(\tau)&=&
\displaystyle\frac{1}{\overline{T}_{NS}}\left(\gamma_G-\mu_G+\displaystyle\frac{\ln(1-\eta_G)}{\tau}\right),\\
&&\\
\tau^\star
&=&-\displaystyle\frac{\ln(1-\eta_G)}{\gamma_G\left(1-\displaystyle\frac{1}{\mathcal{R}_T^G}\right)}.\\
  \end{array}
\right.
\end{equation}
One can note that $\sigma_G^\star$, $\sigma_{NS}^\star$ and
$\tau^\star$ determined regions of stability/instability of forest
and grassland solutions, with respect to $\sigma_G$, $\sigma_{NS}$
and $\tau$ variations.
\end{document}